\documentclass[10pt]{amsart} 
\usepackage[cp1251]{inputenc}
\usepackage{amstext}
\usepackage{amsfonts}
\usepackage{amssymb}
\usepackage{amsbsy}
\usepackage{amsmath}
\usepackage{latexsym}
\usepackage{textcomp}
\usepackage{xy}
\usepackage{hhline}
\xyoption{all}
\newcommand{\vtx}[1]{*+[o][F-]{\scriptscriptstyle #1}} %the command to draw a vertex

\makeatletter %------------------------------Roman numerals----------

\newcommand{\Rmnum}[1]{\expandafter\@slowromancap\romannumeral #1@}
\makeatother  %--------------------------------------------

\newcounter{num}[section] %

\newenvironment{theo}
{\refstepcounter{num}%
\bigskip\noindent{\bf Theorem~\arabic{section}.\arabic{num}. }\it}
\newenvironment{lemma}
{\refstepcounter{num}%
\bigskip\noindent{\bf Lemma~\arabic{section}.\arabic{num}. }\it}
\newenvironment{example}
{\refstepcounter{num}%
\bigskip\noindent{\bf Example~\arabic{section}.\arabic{num}.}}

\newenvironment{remark}
{\refstepcounter{num}%
\bigskip\noindent{\bf Remark~\arabic{section}.\arabic{num}.}}

\newcounter{thepic}

\newenvironment{eq}{\begin{equation}}{\end{equation}}

\newcommand{\si}{\sigma}
\newcommand{\al}{\alpha}
\newcommand{\be}{\beta}
\newcommand{\ga}{\gamma}
\newcommand{\la}{\lambda}
\newcommand{\de}{\delta}
\newcommand{\De}{\Delta}

\newcommand{\LA}{\langle}
\newcommand{\RA}{\rangle}

\newcommand{\ov}[1]{\overline{#1}}
\newcommand{\un}[1]{{\underline{#1}} }
\newcommand{\idmap}{{\rm id}} % the identical map

\newcommand{\tr}{\mathop{\rm tr}}

\newcommand{\alg}{{\rm alg}}
\newcommand{\mdeg}{\mathop{\rm mdeg}}
\newcommand{\diag}{\mathop{\rm diag}}
\newcommand{\Char}{\mathop{\rm char}}

\newcommand{\Ker}{{\mathop{\rm{Ker }}}}

\newcommand{\FF}{{\mathbb{F}}}   % base field
\newcommand{\NN}{{\mathbb{N}}}
\newcommand{\ZZ}{{\mathbb{Z}}}   % integers
\newcommand{\QQ}{{\mathbb{Q}}}

\newcommand{\Q}{\mathcal{Q}}    %Quiver Q
\newcommand{\path}[1]{\mathop{{\rm path}(#1)}} %the set of all paths of Q, including empty paths

\newcommand{\algA}{\mathcal{A}}    %algebra A
\newcommand{\algF}{\mathcal{F}}    %algebra F

\newcommand{\X}{\LA X\RA}
\newcommand{\EX}{\LA \widetilde{X}\RA}
\newcommand{\C}{\mathcal{C}}     % the algebra of concomitants
\newcommand{\Tid}{{\rm T}} % it is letter "T" from "T-ideal"
\newcommand{\Tidskew}{T} % it is letter "T" from "{\it T-ideal}"

\newcommand{\AlgLarge}{\si\EX}  % large free algebra
\newcommand{\KLarge}[1]{\widetilde{K}_{#1}}  % ideal of relations in the large free algebra
\newcommand{\PhiLarge}[1]{\widetilde{\Phi}_{#1}}  % the map from $\AlgLarge$ to $R^{GL(n)}$ 
\newcommand{\TLarge}[1]{\widetilde{T}_{#1}}  % ideal of relations in the large free algebra of concomitants
\newcommand{\PsiLarge}[1]{\widetilde{\Psi}_{#1}}  % the map from the large free algebra of concomitants to $\C_n$ 
\newcommand{\piLarge}{\widetilde{\pi}}  % the map from the absolutely free algebra to $\AlgLarge$ 

\newcommand{\AlgSmall}[1]{\si\X_{#1}}  % small free algebra
\newcommand{\KSmall}[1]{K_{#1}}  % ideal of relations in the small free algebra
\newcommand{\PhiSmall}[1]{\Phi_{#1}}  % the map from $\AlgSmall{n}$ to $R^{GL(n)}$ 
\newcommand{\TSmall}[1]{T_{#1}}  % ideal of relations in the small free algebra of concomitants
\newcommand{\PsiSmall}[1]{\Psi_{#1}}  % the map from the small free algebra of concomitants to $\C_n$ 
\newcommand{\piSmall}[1]{\pi_{#1}}  % the map from the absolutely free algebra to $\AlgSmall$ 

\newcommand{\PhiAbs}[1]{\widehat{\Phi}_{#1}}  % the map from the absolutely free algebra $\si\X$ to $R^{GL(n)}$ 
\newcommand{\PsiAbs}[1]{\widehat{\Psi}_{#1}}  % the map from the absolutely free algebra $\si\X\otimes\FF\X^{\#}$ to $\C_n$ 
\newcommand{\TAbs}[1]{\widehat{T}_{#1}}  % kernel of $\PsiAbs$

\newcommand{\Y}{\LA Y\RA}
\newcommand{\EY}{\LA \widetilde{Y}\RA}
\newcommand{\CY}{\mathcal{C}'}     % the algebra of concomitants for O(n)-invariants

\newcommand{\AlgLargeY}{\si\EY}  % large free algebra
\newcommand{\KLargeY}[1]{\widetilde{K}'_{#1}}  % ideal of relations in the large free algebra
\newcommand{\PhiLargeY}[1]{\widetilde{\Phi}'_{#1}}  % the map from $\AlgLargeY$ to $R^{O(n)}$ 
\newcommand{\TLargeY}[1]{\widetilde{T}'_{#1}}  % ideal of relations in the large free algebra of concomitants
\newcommand{\PsiLargeY}[1]{\widetilde{\Psi}'_{#1}}  % the map from the large free algebra of concomitants to $\CY_n$ 
\newcommand{\piLargeY}{\widetilde{\pi}'}  % the map from the absolutely free algebra to $\AlgLargeY$ 

\newcommand{\AlgSmallY}[1]{\si\Y_{#1}}  % small free algebra
\newcommand{\KSmallY}[1]{K'_{#1}}  % ideal of relations in the small free algebra
\newcommand{\PhiSmallY}[1]{\Phi'_{#1}}  % the map from $\AlgSmallY{n}$ to $R^{O(n)}$ 
\newcommand{\TSmallY}[1]{T'_{#1}}  % ideal of relations in the small free algebra of concomitants
\newcommand{\piSmallY}[1]{\pi'_{#1}}  % the map from the absolutely free algebra to $\AlgSmallY$ 

\newcommand{\PhiAbsY}[1]{\widehat{\Phi}'_{#1}}   % the map from the absolutely free algebra $\si\Y$ to $R^{O(n)}$ 

\newcommand{\pplus}{}   % We have not decided yet how to denote the image of $f\in\si\X$ in $\AlgSmall{n}$: by $f^{+}$ or just by $f$. That is why we need $\pplus$!

\newcommand{\loopR}[3]{%
\begin{picture}(20,0)(#1,#2)
\put(-2,1){\llap{$\scriptstyle #3$}} \put(11,3){\circle{20}} \put(20,6){\vector(1,-4){1}}
\end{picture}}
\newcommand{\loopL}[3]{%
\begin{picture}(20,0)(#1,#2)
\put(22,1){$\scriptstyle #3$} \put(9,3){\circle{20}} \put(0,6){\vector(-1,-4){1}}
\end{picture}}

\begin{document}
\renewcommand{\refname}{References}
\thispagestyle{empty}

\title{Matrix identities with forms}%
\author{{Artem A. Lopatin}}%
\noindent\address{\noindent{}Artem A. Lopatin%
\newline\hphantom{iiii} Omsk Branch of
\newline\hphantom{iiii} Federal State Budgetary Establishment
\newline\hphantom{iiii} Sobolev Institute of Mathematics, SB RAS,
\newline\hphantom{iiii} Pevtsova street, 13,
\newline\hphantom{iiii} 644099, Omsk, Russia%
\newline\hphantom{iiii} http://www.iitam.omsk.net.ru/\~{}lopatin}%
\email{artem\underline{ }lopatin@yahoo.com}%

\vspace{1cm}
\maketitle {\small
\begin{quote}
\noindent{\sc Abstract. }  Consider the algebra $M_n(\FF)$ of $n\times n$ matrices over an infinite field $\FF$ of arbitrary characteristic.  An identity for $M_n(\FF)$ with forms is such a polynomial in $n\times n$ generic matrices and in $\si_k(x)$, $1\leq k\leq n$, coefficients in the characteristic polynomial of monomials in generic matrices, that is equal to zero matrix. This notion is a characteristic free analogue of identities for $M_n(\FF)$ with trace and it can be applied to the problem of investigation of identities for $M_n(\FF)$. In 1996 Zubkov established an infinite generating set for the $\Tid$-ideal $\TSmall{n}$ of identities for $M_n(\FF)$ with forms. Namely, for $t>n$ he introduced partial linearizations of $\si_t$ and proved that they together with the well-known free relations and the Cayley--Hamilton polynomial $\chi_n$ generate $\TSmall{n}$ as a $\Tid$-ideal. We show that it is enough to take partial linearizations of $\si_t$ for $n<t\leq 2n$. In particular, the $\Tid$-ideal $\TSmall{n}$ is finitely based. 

Working over a field of characteristic different from two, we obtain a similar result for the ideal $\TSmallY{n}$ of identities with forms for the $\FF$-algebra generated by $n\times n$ generic and transpose generic matrices. It follows from our previous papers that the $\Tid$-ideal $\TSmallY{n}$ is generated by partial linearizations of $\si_{t,r}$ for $t+2r>n$, the well-known free relations, $\chi_{t,r}$ for $t+2r=n$, and $\zeta_{t,r}$ for $t+2r=n-1$, where $\si_{t,r}$ is the identity introduced by Zubkov in 2005 and $\chi_{t,r}$, $\zeta_{t,r}$ are generalizations of the Cayley--Hamilton polynomial. We prove that  
it is enough to take partial linearizations of $\si_{t,r}$ for $n<t+2r\leq 2n$. In particular, the $\Tid$-ideal $\TSmallY{n}$ is finitely based.

These results imply that ideals of identities for the algebras of matrix $GL(n)$- and $O(n)$-invariants are generated by the well-known free relations together with partial linearizations of $\si_t$ for $n<t\leq 2n$ and partial linearizations of $\si_{t,r}$ for $n<t+2r\leq 2n$, respectively.
\medskip


\end{quote}
}

%==========================================================================
%==========================================================================
%------Sec1--------------------------------------------------
\section{Introduction}\label{section1}

We work over an infinite field $\FF$ of arbitrary characteristic $p=\Char{\FF}\geq0$. All vector spaces, algebras, modules as well as tensor products are over $\FF$ and all algebras are associative with unity unless otherwise stated. All ideals are two-sided.

\subsection{Notations}

Given $n>1$ we consider $n\times n$ {\it generic} matrices $X_k=(x_{ij}(k))_{1\leq i,j\leq d}$ ($k\geq 1$) with entries from the following polynomial algebra
$$R=R_{n}=\FF[x_{ij}(k)\,|\,1\leq i,j\leq n,\, k\geq 1].$$
Denote coefficients in the characteristic polynomial
of an arbitrary $n\times n$ matrix $A$ by $\sigma_t(A)$, i.e., %
$$\det(\lambda E -A )=\sum_{t=0}^{n} (-1)^t\lambda^{n-t}\sigma_t(A).$$
So, $\sigma_0(A)=1$, $\sigma_1(A)=\tr(A)$ and $\sigma_n(A)=\det(A)$.
The algebra of {\it matrix $GL(n)$-invariants} $R^{GL(n)}$ is known to be generated by $\si_t(A)$, where $1\leq t\leq n$ and $A$ is a monomial in generic matrices. Moreover, we can assume that $A$ ranges over {\it primitive} monomials, i.e., $A\neq B^l$ for $l>1$ and a monomial $B$ in generic matrices. The mentioned generators of $R^{GL(n)}$ were found by Sibirskii~\cite{Sibirskii_1968} and Procesi~\cite{Procesi_1976} in characteristic zero case and by Donkin~\cite{Donkin_1992a} in the general case. The formal definition of $R^{GL(n)}$ together with some properties can be found, for example, in~\cite{DKZ_2002}. 

The algebra of {\it $n\times n$ matrices with forms} (or, the algebra of {\it concomitants}) 
$$\C_n=\alg_{\FF}\{X_1,X_2,\ldots,fE\}$$ 
is generated by generic matrices and $fE$, where $f$ ranges over $R^{GL(n)}$ and $E$ stands for the identity $n\times n$ matrix.  The ideal of identities for the algebra $M_n(\FF)$ of $n\times n$ matrices over $\FF$ coincides with the ideal of identities for $\alg_{\FF}\{X_1,X_2,\ldots\}\subset\C_n$. So a description of identities for $\C_n$ can be applied to the problem of investigation of identities for $M_n(\FF)$. Note that the identities for $M_n(\FF)$ are described only in the case of $n=2$ and $p\neq2$ (see~\cite{Razmyslov_1973},~\cite{Koshlukov_2001},~\cite{Koshlukov_2004}). In particular, it is shown that the T-ideal of identities for $M_n(\FF)$ is finitely based in the case of $n=2$ and $p\neq2$, but it is an open problem for $n=p=2$ as well as in the case of $n>2$ and $p>0$.

We define the following notions.
\begin{enumerate}
\item[$\bullet$] Let $\X$ be the semigroup (without unity) freely generated by {\it letters}  $x_1,x_2,\ldots$ and $\X^{\#}=\X\sqcup\{1\}$.

\item[$\bullet$] Let $\FF\X$ and $\FF\X^{\#}$ be the vector spaces with the bases $\X$ and $\X^{\#}$, respectively. Note that elements of $\FF\X$ and $\FF\X^{\#}$ are {\it finite} linear combinations of monomials from $\X$ and $\X^{\#}$, respectively.

\item[$\bullet$] Define a homomorphism of algebras $\phi_n:\FF\X^{\#}\to \alg_{\FF}\{E,X_1,X_2,\ldots\}$ by $1\to E$ and $x_k\to X_k$ for all $k\geq1$.
 \end{enumerate}

Consider a {\it free} algebra $\algF$ for $R^{GL(n)}$, i.e., $\algF$ is a free commutative $\FF$-algebra, equipped with a surjective homomorphism $\Phi_{\algF}:\algF\to R^{GL(n)}$, whose kernel is called the {\it ideal of relations} for $R^{GL(n)}$ with respect to $\algF$. Then the algebra $\algF\otimes \FF\X^{\#}$ is called a {\it free} algebra for $\C_n$ and the kernel of the surjective homomorphism 
$$\Psi_{\algF}:\algF\otimes \FF\X^{\#} \to \C_n,\qquad f\otimes b \to \Phi_{\algF}(f)\, \phi_n(b)$$
is the {\it ideal of relations} for $\C_n$ with respect to $\algF\otimes \FF\X^{\#}$. There are several ways to introduce a free algebra $\algF$ for $R^{GL(n)}$ and, consequently, for $\C_n$. Below we consider 
\begin{enumerate}
\item[$\bullet$] the {\it absolutely} free algebra $\si\X$ for $R^{GL(n)}$, 

\item[$\bullet$] the {\it large} free algebra $\AlgLarge$ for $R^{GL(n)}$ with the ideal of relations $\KLarge{n}$, 

\item[$\bullet$] the {\it small} free algebra $\AlgSmall{n}$ for $R^{GL(n)}$ with the ideal of relations $\KSmall{n}$,

\item[$\bullet$] the large and small free algebras $\AlgLarge \otimes \FF\X^{\#}$ and $\AlgSmall{n} \otimes \FF\X^{\#}$, respectively, for $\C_n$ with the ideals of relations $\TLarge{n}$ and $\TSmall{n}$, respectively.
\end{enumerate}

Our main results are the following ones:
\begin{enumerate}
\item[$\bullet$] the ideals of relations $\KSmall{n}$ and $\TSmall{n}$ are finitely based (see Theorem~\ref{theo_GL_main});

\item[$\bullet$] the ideals $\KLarge{n}$ and $\TLarge{n}$ are finitely based if and only if $p=0$ (see Lemma~\ref{lemma_GL_fb});

\item[$\bullet$] similar results are obtained in case $p\neq2$ for the ideal of identities with forms of the $\FF$-algebra generated by $n\times n$ generic and transpose generic matrices (see Theorem~\ref{theo_O_main} and Lemma~\ref{lemma_O_fb}).
\end{enumerate}  

Let us determine these free algebras.

\begin{enumerate}
\item[$\bullet$] Introduce the natural lexicographical linear order on $\X$ by setting $x_1>x_2>\cdots$ and $ab>a$ for $a,b\in\X$. (Note that we can actually consider any other lexicographical linear order).

\item[$\bullet$] Let $\AlgSmall{n}$ ($\si\X$, respectively) be a ring with unity of commutative polynomials over $\FF$ freely generated by ``symbolic'' elements $\si_t(a)$, where $1\leq t\leq n$ ($t\geq1$, respectively) and $a$ ranges over polynomials from $\FF\X$ with coefficient $1$ in the highest term with respect to the introduced lexicographical order on $\X$. Define 
$$\si_t(\al a)=\al^t\si_t(a)$$
for $\al\in\FF$ and denote $\si_0(a)=1$, $\tr(a)=\si_1(a)$. Note that $\si_t(0)=0$ and $\AlgSmall{n}\subset \si\X$.

\item[$\bullet$] We say that $a,b\in\X$ are {\it cyclic equivalent} and write $a\stackrel{c}{\sim} b$
if $a=a_1a_2$ and $b=a_2a_1$ for some $a_1,a_2\in\X^{\#}$.

\item[$\bullet$] Let $\EX\subset\X$ be a subset of maximal (with respect to the introduced lexicographical order on $\X$) representatives of $\stackrel{c}{\sim}$-equivalence classes of {\it primitive} elements, i.e., for $a\in\EX$ we have $a\neq b^l$ for all $b\in\X$ and $l>1$.

\item[$\bullet$] Assume that $\AlgLarge$ is a ring with unity of commutative polynomials over $\FF$ freely generated by ``symbolic'' elements $\si_t(a)$, where $t>0$ and $a\in\EX$. 
\end{enumerate}

There are the following maps between the defined free algebras. By Lemma~\ref{lemma_GL_Donkin}, we have the surjective homomorphism $\piLarge:\si\X\to\AlgLarge$. Define a surjective homomorphism $\piSmall{n}:\si\X\to\AlgSmall{n}$ by 
$$\piSmall{n}(\si_t(a))=\left\{
\begin{array}{rl}
\si_t(a),& 1\leq t\leq n\\
0,& t> n\\
\end{array}
\right..$$ 
Consider the surjective homomorphism 
$$\PhiAbs{n}:\si\X\to R^{GL(n)}$$ 
such that $\si_t(a) \to \si_t(\phi_n(a))$ for $1\leq t\leq n$ and $\si_t(a) \to 0$ for $t>n$, where $a\in\FF\X$. Since
$$\si_t(\al A)=\al^t\si_t(A)$$
holds for an arbitrary $n\times n$ matrix $A$ over a commutative $\FF$-algebra and $1\leq t\leq n$, the homomorphism $\PhiAbs{n}$ is well-defined. Similarly, we define surjective homomorphisms 
$$\PhiLarge{n}:\AlgLarge\to R^{GL(n)}\;\text{ and }\;\PhiSmall{n}:\AlgSmall{n}\to R^{GL(n)}.$$ 
Its kernels $\KLarge{n}$ and $\KSmall{n}$, respectively, are the ideals of relations for $R^{GL(n)}$ in the large and small free algebra, respectively. Then it is well-known that the following diagram is commutative. Namely, its left triangle is commutative by the definition and its right triangle is commutative by Remark~\ref{remark_GL_diagram} (see below). 
$$
\begin{picture}(0,120)
\put(0,95){%
\put(0,-2){\vector(0,-1){58}}%
\put(15,0){\vector(3,-2){35}}%
\put(-15,0){\vector(-3,-2){35}}%
\put(-11,5){$\si\X$}%
\put(-75,-33){$\AlgSmall{n}$}%
\put(50,-33){$\AlgLarge$}%
\put(50,-40){\vector(-3,-2){35}}%
\put(-50,-40){\vector(3,-2){35}}%
\put(110,0){\vector(-3,-2){35}}%
\put(-110,0){\vector(3,-2){35}}%
\put(-6,-75){$R^{GL(n)}$}%
\put(-125,5){$\KSmall{n}$}%
\put(115,5){$\KLarge{n}$}%
\put(3,-33){$\scriptstyle\PhiAbs{n}$}%
\put(-40,-8){$\scriptstyle\piSmall{n}$}%
\put(33,-8){$\scriptstyle\piLarge$}%
\put(-35,-48){$\scriptstyle\PhiSmall{n}$}%
\put(25,-48){$\scriptstyle\PhiLarge{n}$}%
\put(-20,-90){\text{Diagram 1.}}%
}%
\end{picture}
$$%
\noindent{}The homomorphisms $\PhiLarge{n}$ and $\PhiSmall{n}$ induce surjective homomorphisms 
$$\PsiLarge{n}=\PhiLarge{n}\otimes\phi_n:\AlgLarge \otimes \FF\X^{\#} \to \C_n\;\text{ and }\;
\PsiSmall{n}=\PhiSmall{n}\otimes\phi_n:\AlgSmall{n}\otimes \FF\X^{\#} \to \C_n,$$
respectively. Its kernels $\TLarge{n}$ and $\TSmall{n}$, respectively, are the ideals of relations for $\C_n$ in the corresponding free algebras. For short, we write $\si_t(a)b$ for $\si_t(a)\otimes b$. We can depict the introduced maps as follows: 
$$ 
\begin{picture}(0,120)
\put(0,95){%
\put(-100,-33){$\AlgSmall{n}\otimes \FF\X^{\#}$}%
\put(35,-33){$\AlgLarge\otimes \FF\X^{\#}$}%
\put(50,-40){\vector(-3,-2){35}}%
\put(-50,-40){\vector(3,-2){35}}%
\put(110,0){\vector(-3,-2){35}}%
\put(-110,0){\vector(3,-2){35}}%
\put(-6,-75){$\C_n$}%
\put(-125,5){$\TSmall{n}$}%
\put(115,5){$\TLarge{n}$}%
\put(-35,-48){$\scriptstyle\PsiSmall{n}$}%
\put(25,-48){$\scriptstyle\PsiLarge{n}$}%
\put(-20,-90){\text{Diagram 2.}}%
}%
\end{picture}
$$%

We say that an ideal $J$ of $\AlgSmall{n} \otimes \FF\X^{\#}$ is a {\it $\Tidskew$-ideal} if it is stable with respect every endomorphism $\varphi$ preserving $\si_t$, i.e., 
$$\varphi(\si_t(a)b)=\varphi(\si_t(a))\varphi(b)\text{ and }\varphi(\si_t(a))=\si_t(\varphi(a))$$
for all $a,b\in\FF\X$. These endomorphisms are determined by substitutions $x_k\to a_k$, where $a_k\in \FF\X$, $k>0$, and we call them {\it substitution} endomorphisms. A $\Tid$-ideal $J$ is {\it finitely based} if it is generated by a finite set $f_1,\ldots,f_s$ as $\Tid$-ideal, i.e., the ideal $J$ is generated by $\varphi(f_1),\ldots,\varphi(f_s)$, where $\varphi$ ranges over substitution endomorphisms.  Similarly, we define the notion of a $\Tid$-ideal for $\AlgLarge \otimes \FF\X^{\#}$, $\AlgLarge$, and $\AlgSmall{n}$. Obviously, $\KLarge{n}$, $\KSmall{n}$, $\TLarge{n}$, and $\TSmall{n}$ are $\Tid$-ideals.

\subsection{Results for $\C_n$}

In case $p=0$ Razmyslov~\cite{Razmyslov_1974} and Procesi~\cite{Procesi_1976} showed that the $\Tid$-ideal $\KLarge{n}$ is generated by a single identity. In particular, $\KSmall{n}$, $\TLarge{n}$, $\TSmall{n}$ are finitely based in characteristic zero case.  In case $p>n$ results of Samoilov~\cite{Samoilov_2007} imply that $\KSmall{n}$ and $\TSmall{n}$ are finitely based. In the case of arbitrary characteristic Zubkov~\cite{Zubkov_1996} described an infinite generating set for the $\Tid$-ideal $\KLarge{n}$ (see Theorem~\ref{theo_Zubkov}) and, therefore, for the ideals $\KSmall{n}$, $\TLarge{n}$, $\TSmall{n}$. 

In our main result we established a finite generating sets for the $\Tid$-ideals $\KSmall{n}$ and $\TSmall{n}$ (see Theorem~\ref{theo_GL_main} and Remark~\ref{remark_theo_GL_main}). In particular, $\KSmall{n}$ and $\TSmall{n}$ are finitely based. Necessary definitions are given in Section~\ref{section2}. To prove Theorem~\ref{theo_GL_main}, in Section~\ref{section3} we obtained an essentially smaller than in~\cite{Zubkov_1996} generating set for $\KLarge{n}$ (see Theorem~\ref{theo_GL} and Remark~\ref{remark_theo_GL}). We also showed that $\KLarge{n}$ and $\TLarge{n}$ are finitely based if and only if $p=0$ (see Lemma~\ref{lemma_GL_fb}). Applying Theorem~\ref{theo_GL}, we completed the proof of Theorem~\ref{theo_GL_main} in Section~\ref{section4}.

\subsection{Results for $\CY_n$}

Assume that $p\neq2$. In Section~\ref{section5} we consider identities with forms for the $\FF$-algebra generated by $n\times n$ generic and transpose generic matrices, or, equivalently, identities for the algebra $\CY_n$ generated $X_i$, $X_i^T$, $fE$, where $i>0$ and $f$ ranges over the algebra $R^{O(n)}$ of {\it matrix $O(n)$-invariants}. A description of identities for $\CY_n$ can be applied to the problem of investigation of identities with transpose involution for $M_n(\FF)$. Note that the identities with transpose involution for $M_n(\FF)$ are described only in the case of $n=2$ and $p\neq2$ (see~\cite{Koshlukov_2005}). 

Similarly to $\AlgLarge$ we introduce {\it large} free algebra $\AlgLargeY$ for $R^{O(n)}$ with the ideal of relations $\KLargeY{n}$. And similarly to $\AlgSmall{n}$ we introduce {\it small} free algebra $\AlgSmallY{n}$ for $R^{O(n)}$ with the ideal of relations $\KSmallY{n}$. Finally, similarly to $\TLarge{n}$ and $\TSmall{n}$ we define ideals of relations $\TLargeY{n}$ and $\TSmallY{n}$ for $\CY_n$ in the large and small free algebras $\AlgLargeY\otimes \FF\Y^{\#}$ and $\AlgSmallY{n}\otimes \FF\Y^{\#}$, respectively. 

In case $p=0$ Procesi~\cite{Procesi_1976} described a finite generating set for the $\Tid$-ideal $\KLargeY{n}$. In particular, $\KSmallY{n}$, $\TLargeY{n}$, $\TSmallY{n}$ are finitely based in characteristic zero case. In the case of arbitrary characteristic an infinite generating set for the $\Tid$-ideal $\KLargeY{n}$ was described in~\cite{Lopatin_free},~\cite{Lopatin_Orel} (see Theorem~\ref{theo_Lopatin}).

We established a finite generating sets for the $\Tid$-ideals $\KSmallY{n}$ and $\TSmallY{n}$ (see Theorem~\ref{theo_O_main}). In particular, $\KSmallY{n}$ and $\TSmallY{n}$ are finitely based. Necessary definitions are given in Section~\ref{section5}. To prove Theorem~\ref{theo_O_main}, in Section~\ref{section6} we obtained an essentially smaller than in~\cite{Lopatin_free} generating set for $\KLargeY{n}$ (see Theorem~\ref{theo_O} and Remark~\ref{remark_theo_O}). We also showed that $\KLargeY{n}$ and $\TLargeY{n}$ are finitely based if and only if $p=0$ (see Lemma~\ref{lemma_O_fb}). Applying Theorem~\ref{theo_O}, we completed the proof of Theorem~\ref{theo_O_main} in Section~\ref{section7}. Note that the proof of Theorem~\ref{theo_O_main} uses the same approach as the proof of Theorem~\ref{theo_GL_main}, but it is essentially more difficult. Namely, instead of core Lemmas~\ref{lemma_GL_sets}, \ref{lemma_GL_key} in case of Theorem~\ref{theo_GL_main} we need 
Lemmas~\ref{lemma_O_sets1}, \ref{lemma_O_key1}, \ref{lemma_O_sets2}, \ref{lemma_O_key2} to prove Theorem~\ref{theo_O_main}.

%------Rem1.1--------------------------------------------------
\begin{remark}
The notion algebra of matrix $GL(n)$-invariants $R_{n,d}^{GL(n)}$ from  papers~\cite{Donkin_1992a},  \cite{Procesi_1976}, \cite{Razmyslov_1974}, \cite{Sibirskii_1968}, \cite{Zubkov_1996} is slightly different from ours. Namely, the algebra $R_{n,d}^{GL(n)}$ from the mentioned papers is generated by $\si_t(A)$, where $1\leq t\leq n$ and $A$ is a monomial in $X_1,\ldots,X_d$. Since $R^{GL(n)}=\bigcup_{d>0} R_{n,d}^{GL(n)}$, part~2 of Theorem~\ref{theo_Zubkov} holds for $R^{GL(n)}$. Similar remark also holds for the algebra of matrix $O(n)$-invariants $R^{O(n)}$ from Section~\ref{section5}.
\end{remark}

%========================================================================================
%========================================================================================
%------Sec2--------------------------------------------------
\section{Relations}\label{section2}

Denote $\NN=\{1,2,\ldots\}$ and $\NN_0=\NN\sqcup \{0\}$. Given $\un{t}=(t_1,\ldots,t_u)\in\NN^u$, we write $|\un{t}|$ for $t_1+\cdots+t_u$ and $\#\un{t}$ for $u$. For short, we write $1^t$ for $(1,\ldots,1)$ ($t$ times). 

Let $\algA=\bigoplus_{k\in\NN_0} \algA_k$ be a graded algebra with $\algA_0=\FF$, $f,h,h_1,\ldots,h_r\in\algA$, and $J\vartriangleleft\algA$ be an ideal. We say that the relation $f=h$ belongs to the ideal $J$ (or, equivalently, holds modulo $J$) if $f-h\in J$. We also say that the relation $f=0$ follows from relations $h_1=0,\ldots,h_r=0$ if $f$ belongs to the ideal generated by $h_1,\ldots,h_r$.  The relation $f=0$ is said to belong to $J$ modulo relations $h_1=0,\ldots,h_r=0$ if $f$ belongs to the ideal generated by $J,h_1,\ldots,h_r$. If $f=\sum_{i=1}^r \al_i f_i h_i$, where $\al_i\in\FF$ and $f_i,h_i\in\algA$ are homogeneous elements of positive degree ($1\leq i\leq r$), then we write $f\equiv0$. If $f-\sum_{i=1}^r \al_i f_i h_i$ belongs to $J$, where $\al_i,f_i,h_i$ are the same as above, then we say that $f\equiv0$ holds modulo $J$.  

For $f=\si_t(a)\in\si\X$ with $a\in\X$ we set $\deg(f)=t\deg(a)$ and $\deg_x(f)=t\deg_x(a)$, where $x$ is a letter and $\deg_x(a)$ stands for a degree of the monomial $a$ in the letter $x$. In the same way we define a degree for elements of $\AlgLarge$ and $\AlgSmall{n}$.  Denote the multidegree of $a\in\X$ by $\mdeg(a)=(\de_1,\de_2,\ldots)$, where $\de_i=\deg_{x_i}(a)$. For short, we write $\mdeg(a)=(\de_1,\ldots,\de_d)$ in case $\de_i=0$ for all $i>d$.

We use notation
$\{\ldots\}_{m}$ for {\it multisets}, i.e., given an equivalence $=$ on a set $S$ and
$a_1,\ldots,a_r,b_1,\ldots,b_s\in S$, we write $\{a_1,\ldots,a_r\}_{m} = \{
b_1,\ldots,b_s\}_{m}$ if and only if $r=s$ and
$$\#\{1\leq j\leq r\,|\,a_j=a_i\}=\#\{1\leq j\leq r\,|\,b_j=a_i\}$$ %
for all $1\leq i\leq r$. We also refer to $\{a_1,\ldots,a_r\}_{m}$ as a {\it multisubset} of $S$.

Consider some relations for $\C_n$ and $R^{GL(n)}$. Given $\un{t}\in\NN_0^u$, we denote by $\Omega(\un{t})$ the set of multisets  
$$\omega=\{\underbrace{e_1,\ldots,e_1}_{k_1},\ldots,\underbrace{e_q,\ldots,e_q}_{k_q}\}_{m}
$$%
such that
\begin{enumerate}
\item[$\bullet$] $e_1,\ldots,e_q \in\EX$ are pairwise different and $k_1,\ldots,k_q\in\NN$ ($q>0$); 

\item[$\bullet$] $k_1\mdeg(e_1)+\cdots+k_q\mdeg(e_q)=\un{t}$.
\end{enumerate}
We set $\si(\omega)=(-1)^{k_1+\cdots+k_q} \si_{k_1}(e_1)\cdots\si_{k_q}(e_q)$. 
For $\un{x}=(x_1,\ldots,x_u)$ we define $\si_{\un{t}}(\un{x})\in \si\X$ as follows:
\begin{eq}\label{eq1}
\si_{\un{t}}(\un{x})= (-1)^{|\un{t}|}\!\!\! \sum_{\omega\in \Omega(\un{t})}  \si(\omega).
\end{eq}%
If $\Omega(\un{t})$ is empty, then we set $\si_{\un{t}}(\un{x})=1$. For $t>0$ denote 
$$F_t(\un{x})=\sum\si_{\un{t}}(\un{x}),$$
where the sum is taken over all $\un{t}\in\NN_0^u$ with $|\un{t}|=t$. For  $\un{a}=(a_1,\ldots,a_u)$ with $a_1,\ldots,a_u\in\FF\X$ we set that $\si_{\un{t}}(\un{a})$ and $F_t(\un{a})$ are the results of substitutions $x_1\to a_1,\ldots,x_u\to a_u$ in $\si_{\un{t}}(\un{x})$ and $F_t(\un{x})$, respectively. By Amitsur's formula~\cite{Amitsur_1980}, for $1\leq t\leq n$ we have that
\begin{eq}\label{eq_Amitsur}
\si_t(a_1+\cdots+a_u)=F_t(\un{a}) 
\end{eq}%
is a relation for $R^{GL(n)}$, i.e., belongs to the kernel of $\PhiAbs{n}$.

%------Ex2.1--------------------------------------------------
\begin{example}\label{ex_21} Taking the image of relation~(\ref{eq_Amitsur}) in $R^{GL(n)}$ we obtain that for an arbitrary $n\times n$ matrices $A,B$ over a commutative $\FF$-algebra the following equalities hold: 
\begin{enumerate}
\item[$\bullet$] $\si_2(A+B)=\si_2(A)+\si_2(B)+\tr(A)\tr(B)-\tr(AB)$,
\item[$\bullet$] $\si_3(A+B)=\si_3(A)+\si_3(B)+\si_2(A)\tr(B)-\tr(AB)\tr(A)+\tr(A^2B)$ 

$\qquad\qquad\qquad\qquad\qquad\quad$ $+\,\si_2(B)\tr(A)-\tr(AB)\tr(B)+\tr(B^2A)$.
\end{enumerate}
\end{example}
\bigskip

For $t\geq1$, $l\geq2$, and an $n\times n$ matrix $A$ over a commutative $\FF$-algebra we have the following well-known formula:
\begin{eq}\label{eq_P}
\si_t(A^l)=\sum\limits_{i_1,\ldots,i_{t l}\geq0}\be^{(t,l)}_{i_1,\ldots,i_{t l}}
    \si_1(A)^{i_1}\cdots\si_{t l}(A)^{i_{t l}},%
\end{eq}%
where we assume that $n\geq tl$ is large enough. 
Denote the right hand side of~(\ref{eq_P}) by $P_{t,l}(A)$. 
In~(\ref{eq_P}) coefficients $\be^{(t,l)}_{i_1,\ldots,i_{rl}} \in \ZZ_p\simeq \ZZ/p\ZZ$  
do not depend on $A$ and $n$.
If we take a diagonal matrix $A=\diag(\al_1,\ldots,\al_n)$, $\al_i\in\FF$, then $\si_t(A^l)$ is a symmetric polynomial in $\al_1,\ldots,\al_n$ and $\si_k(A)$ is the
$k^{\rm th}$ elementary symmetric polynomial in $\al_1,\ldots,\al_n$, where $1\leq k\leq n$. Thus the coefficients
$\be^{(t,l)}_{i_1,\ldots,i_{tl}}$ can easily be found. Some information about the polynomial $P_{t,l}(A)$ is given in Lemma~\ref{lemma_P} (see below).

%------Ex2.2----------------------------------------------------
\begin{example} We have the following partial cases of formula~(\ref{eq_P}):
\begin{enumerate}
\item[$\bullet$] $\tr(A^2)=\tr(A)^2-2\si_2(A)$,

\item[$\bullet$] $\tr(A^3) = \tr(A)^3 - 3\si_2(A)\tr(A) + 3\si_3(A)$,

\item[$\bullet$] $\tr(A^4) = \tr(A)^4 - 4\si_2(A)\tr(A)^2 + 2\si_2(A)^2 + 4\si_3(A)\tr(A) - 4\si_4(A)$,

\item[$\bullet$] $\si_2(A^2) = \si_2(A)^2 - 2\si_3(A)\tr(A) + 2\si_4(A)$.
\end{enumerate}
\end{example}

%------Rem2.3--------------------------------------------------
\begin{remark}\label{remark_GL_notations} Let $f\in\si\X$. Taking the image of $f$ with respect to $\piLarge$ ($\piSmall{n}$, respectively), we can consider $f$ as an element of $\AlgLarge$ ($\AlgSmall{n}$, respectively). As an example, let $f_k=\si_k(x+y)\in\si\X$ for $k\geq1$ and letters $x\neq y$. Then $f_2$ in $\AlgLarge$ is $\si_2(x)+\si_2(y)+\tr(x)\tr(y)-\tr(xy)$. On the other hand, $f_3$ in $\AlgSmall{n}$ is zero in case $n=2$ and $f_3$ in $\AlgSmall{n}$ is $\si_3(x+y)$ in case $n\geq 3$.
\end{remark}
\bigskip

The next remark follows from Lemma~\ref{lemma_GL_Donkin} (see below) and the definition of $\AlgLarge$.

%------Rem2.4--------------------------------------------------
\begin{remark}\label{remark_GL_OK1}
Assume that $\un{t}\in\NN^u$ and $t=|\un{t}|$. Then $\si_{\un{t}}(x_1,\ldots,x_u)\in\AlgLarge$ is a {\it partial linearization} of $\si_t(x_1)$, i.e., it is the coefficient of $\la_1^{t_1}\cdots \la_u^{t_u}$ in $\si_t(\la_1 x_1+\cdots+\la_u x_u)\in\AlgLarge$ considered as a polynomial in $\la_1,\ldots,\la_u\in\FF$.

Moreover, for $\un{k}\in\NN^l$ with $k=|\un{k}|$ we have that $\si_{\un{t},\un{k}}(x_1,\ldots,x_{u+l})\in\AlgLarge$ is a partial linearization of $\si_{(\un{t},k)}(x_1,\ldots,x_{u},x_{u+1})$.
\end{remark}
\bigskip

As we have mentioned in Section~\ref{section1}, we will usually omit $\otimes$ in the elements of $\si\X \otimes \FF\X^{\#}$. Given $a\in\FF\X$ and $t\geq0$, let  $\chi_t(a)\in\si\X \otimes \FF\X^{\#}$ be the Cayley--Hamilton polynomial, i.e, 
\begin{eq}\label{eq_chi}
 \chi_t(a)=\sum_{i=0}^t (-1)^i \si_i(a) a^{t-i}.
\end{eq}%
Note that $\chi_{0}(a)=1$. As in Remark~\ref{remark_GL_notations}, we can consider $\chi_{t}(a)$ as an element of $\AlgLarge\otimes \FF\X^{\#}$ as well as of $\AlgSmall{n}\otimes \FF\X^{\#}$. The Cayley--Hamilton theorem implies that 
$$\chi_n(a)=0$$
is a relation for $\C_n$, i.e., belongs to $\TLarge{n}$ and $\TSmall{n}$.  The proof of the following Theorem~\ref{theo_GL_main} and Remark~\ref{remark_theo_GL_main} is given in Section~\ref{section4}.

%------Th2.5--------------------------------------------------
\begin{theo}\label{theo_GL_main}
\begin{enumerate}
\item[1.] The ideal of relations $\TSmall{n}$ for $\C_n$ is generated by $\KSmall{n}\otimes 1$ and 
$\chi_n(a)=0$ for $a\in\FF\X$.  

\item[2.] The ideal of relations $\KSmall{n}$ for $R^{GL(n)}\simeq \AlgSmall{n} / \KSmall{n}$ is generated by 
\begin{enumerate}
\item[(a)] $\si_t(a+b)=F_t(a,b)$ for $1\leq t\leq n$, where $a,b\in\FF\X$;  

\item[(b)] $\si_t(a^l)=P^{\pplus}_{t,l}(a)$ for $1\leq t\leq n$, $1<l\leq n$, where $a\in\X$;

\item[(c)] $\si_t(ab)=\si_t(ba)$ for $1\leq t\leq n$, where $a,b\in\X$;

\item[(d)] $\si^{\pplus}_{\un{t}}(a_1,\ldots,a_u)=0$ for $n<|\un{t}|\leq 2n$, where $\un{t}\in\NN^u$, $u>1$, and $a_i\in\X$ for all $i$.%
 \end{enumerate}
 \end{enumerate}
In particular, ideals $\TSmall{n}$ and $\KSmall{n}$ are finitely based.
\end{theo}
\bigskip

Relations (a), (b), (c) from Theorem~\ref{theo_GL_main} are called {\it free} relations, because, being considered as elements of $\si\X$, they belong to the kernel of $\PhiAbs{n}$ for all $t\geq 1$, $l>1$ and do not depend on $n$.

%------Rem2.6--------------------------------------------------
\begin{remark}\label{remark_theo_GL_main}
In the formulation of Theorem~\ref{theo_GL_main} we can assume that $\un{t}\in\NN^u$ from relation~(d) satisfies the following conditions:
\begin{eq}\label{eq_cond1}
t_1\geq \cdots \geq t_u,
\end{eq}
\vspace{-0.5cm}%
\begin{eq}\label{eq_cond2}
t_1,\ldots,t_u\in\{1,p,p^2,p^3,\ldots\},
\end{eq}
\vspace{-0.5cm}%
\begin{eq}\label{eq_cond3}
\text{either } |\un{t}|=n+1, \text{ or } n+1<|\un{t}|\leq 2n \text{ and } |\un{t}|-\min\{t_i\}\leq n. %
\end{eq}%
In particular, if the second case from~(\ref{eq_cond3}) holds, then $t_i\neq 1$ for all $i$. These conditions enable us to diminish the number of multidegrees $\un{t}$ from~(d) considerably. Namely, it is not difficult to see that conditions~(\ref{eq_cond1}), (\ref{eq_cond2}), (\ref{eq_cond3}) imply that
\begin{enumerate}
\item[$\bullet$] if $p=0$ or $p>n$, then $\un{t}=1^{n+1}$;

\item[$\bullet$] if $\frac{n}{2}<p\leq n$, then $\un{t}$ belongs to the following list: $1^{n+1}$, $(p,1^{n+1-p})$, $(p,p)$;

\item[$\bullet$] if $\frac{n}{3}<p\leq \frac{n}{2}$ and $p\neq2$, then $\un{t}$ belongs to the following list: $1^{n+1}$, $(p,1^{n+1-p})$, $(p,p,1^{n+1-2p})$, $(p,p,p)$.
\end{enumerate}  
\end{remark}
\bigskip

Note that in the formulation of Theorem~\ref{theo_GL_main} we can not consider elements $\si_t(a)$ for $n<t\leq 2n$, $a\in\FF\X$ instead of relations (d), because images of these elements in $\AlgSmall{n}$ are zeros.

%==================================================================================
%==================================================================================
%------Sec3--------------------------------------------------
\section{Large free algebra of $GL(n)$-invariants}\label{section3}

We start this section with the known description of the ideal of relations $\KLarge{n}$.

%-----Th3.1---------------------------------------------------------
\begin{theo}(Zubkov~\cite{Zubkov_1996})\label{theo_Zubkov} 
\begin{enumerate}
\item[1.] The ideal of relations $\TLarge{n}$ for $\C_n$ is generated by $\KLarge{n}\otimes 1$ and $\chi_n(a)=0$ for $a\in\FF\X$.  

\item[2.] The ideal of relations $\KLarge{n}$ for $R^{GL(n)}\simeq \AlgLarge / \KLarge{n}$ is generated by $\si_t(a)=0$ for $t>n$ and $a\in\FF\X$.
\end{enumerate}
\end{theo} 
\bigskip

The next lemma describes the large free algebra $\AlgLarge$ as a quotient of the absolutely free algebra $\si\X$.

%------L3.2------------------------------------------------------------
\begin{lemma}\label{lemma_GL_Donkin}(Donkin~\cite{Donkin_1993a}) We have $\AlgLarge\simeq \si\X/ L$ for the ideal $L$ generated by
\begin{enumerate}
\item[(a)] $\si_t(a_1+\cdots+a_u)=F_{t}(a_1,\ldots,a_u)$, 

\item[(b)] $\si_t(a^l)=P_{t,l}(a)$,

\item[(c)] $\si_t(ab)=\si_t(ba)$, 
\end{enumerate}
where $t>0$, $l,u>1$, $a_1,\ldots,a_u\in \FF\X$, and $a,b\in\X$. 
\end{lemma}
\bigskip

In this section we prove the following theorem together with Remark~\ref{remark_theo_GL}.

%-----Th3.3---------------------------------------------------
\begin{theo}\label{theo_GL}
The ideal of relations $\KLarge{n}$ for $R^{GL(n)}\simeq \AlgLarge / \KLarge{n}$ is generated by 
\begin{enumerate}
\item[$\bullet$] $\si_{t}(a)=0$, where $n<t\leq 2n$ and $a\in\FF\X$;% 

\item[$\bullet$] $\si_t(b)=0$, where $t>2n$ and $b\in\EX$. 
\end{enumerate}
\end{theo}

%------Rem3.4--------------------------------------------------
\begin{remark}\label{remark_theo_GL}
We can reformulate Theorem~\ref{theo_GL} as follows: the ideal $\KLarge{n}$ is generated by 
\begin{enumerate}
\item[$\bullet$] $\si_{\un{t}}(a_1,\ldots,a_u)=0$, where $\un{t}\in\NN^u$ ($u>1$) satisfies conditions~(\ref{eq_cond1}), (\ref{eq_cond2}), (\ref{eq_cond3}) and $a_i\in\X$ for $1\leq i\leq u$; 

\item[$\bullet$] $\si_t(b)=0$, where $t>n$ and $b\in\EX$. 
\end{enumerate}
\end{remark}
\bigskip

We split the proof of Theorem~\ref{theo_GL} and Remark~\ref{remark_theo_GL} into several lemmas. Denote by $J_t$ the ideal of $\AlgLarge$ generated by $\si_t(a)$, $a\in\FF\X$. Since the field $\FF$ is infinite, Remark~\ref{remark_GL_OK1} implies that elements $\si_{\un{t}}(a_1,\ldots,a_u)$ generate the ideal $J_{t}$ for $t=|\un{t}|$, where $\un{t}\in\NN^u$ and $a_1,\ldots,a_u\in\X$. We write $J_t^{(p)}$ for the $\FF$-subspace of $\AlgLarge$ spanned by $\si_{\un{t}}(a_1,\ldots,a_u)$ for $\un{t}\in\NN^u$ satisfying $t=|\un{t}|$, $t_i\in\{1,p,p^2,\ldots\}$ and $a_i\in\X$ for all $i$. 

The key idea of the proof of Theorem~\ref{theo_GL} is the fact that $\si_{(k,t)}(a,b)\in J_t$ (see Lemma~\ref{lemma_GL_key}). This fact together with Lemma~\ref{lemma_GL_key2} enables us to show that $\si_{\un{t}}(a_1,\ldots,a_u)$, where $|\un{t}|>n$, $u>1$, $a_1,\ldots,a_u\in\X$, belongs to the ideal of $\AlgLarge$, generated by elements from Theorem~\ref{theo_GL}.

%------L3.5--------------------------------------------------
\begin{lemma}\label{lemma_P}
For a letter $x$ we have that 
\begin{enumerate} 
\item[1)] every summand of $P_{t,l}(x)\in\si\X$ contains a multiple $\si_k(x)$ with $k\geq t$; in particular, if $t>n$, then the image of every relation (b) from Lemma~\ref{lemma_GL_Donkin} with respect to $\piSmall{n}$ is zero; 

\item[2)] if $p>0$, $t=p^r$, $l=p^s$ for $r\geq0$, $s>0$, then $P_{t,l}(x)=\si_t(x)^l$.
\end{enumerate}
\end{lemma}
\begin{proof}
1) We can assume that $t>1$. Let $P_{t,l}(x)=P+Q$ in $\si\X$, where $P,Q$ are polynomials in $\si_i(x)$, $i>0$, and every summand of $P$ ($Q$, respectively) contains $\si_k(x)$ for some $k\geq t$ (does not contain $\si_k(x)$ for any $k\geq t$, respectively). Let $Q$ be a non-zero polynomial. We set $n=t-1$. By part~2 of Theorem~\ref{theo_Zubkov} and Lemma~\ref{lemma_GL_Donkin}, $\piLarge(Q)$ lies in $\KLarge{n}$. Acting by $\PhiLarge{n}$, we obtain a non-trivial relation between $\tr(X),\si_2(X),\ldots, \si_n(X)$, where $X$ is the generic $n\times n$ matrix corresponding to the letter $x$. But it is well-known that these elements are algebraically independent over $\FF$; a contradiction.  

\medskip
2) It follows from $(\al_1+\cdots+\al_m)^p=\al_1^p+\cdots+\al_m^p$ for $\al_1,\ldots,\al_m\in\FF$, $m>0$  and the reasoning after formula~(\ref{eq_P}).
\end{proof}

%------Rem3.6------------------------------------------------------------
\begin{remark}\label{remark_GL_diagram} The right triangle of Diagram~1 is commutative. To show this we use definitions of $\PhiAbs{n}$ and $\PhiLarge{n}$ together with the claim that $\PhiAbs{n}$ sends relations~(a), (b), (c) of Lemma~\ref{lemma_GL_Donkin} to zero. In case $t>n$ this claim follows from part~2 of Theorem~\ref{theo_Zubkov} and part~1 of Lemma~\ref{lemma_P}, and in case $1\leq t\leq n$ see Section~\ref{section2}. 
\end{remark}
\bigskip

%------L3.7--------------------------------------------------
\begin{lemma}\label{lemma_GL_OK2}
Assume that $\un{t}\in\NN^u$ and $t=|\un{t}|$. Then $t_1\!!\si_{\un{t}}(x_1,\ldots,x_u)=\si_{\un{t}'}(\underbrace{x_1,\ldots,x_1}_{t_1},x_2,\ldots,x_u)$ in $\AlgLarge$, where $\un{t}'$ stands for $(1^{t_1}\!,t_2,\ldots,t_u)$.
\end{lemma}
\begin{proof} We work in $\AlgLarge$.
 Remark~\ref{remark_GL_OK1} implies that $\si_{\un{t}'}(\underbrace{x_1,\ldots,x_1}_{t_1},x_2,\ldots,x_u)$ is equal to the coefficient of $\la_1\cdots \la_{t_1}$ in $\si_{\un{t}}(\la_1 x_1+\cdots+\la_{t_1} x_1, x_2, \ldots, x_u) = (\la_1+\cdots+\la_t)^{t_1} \si_{\un{t}}(x_1,\ldots,x_u)$, where $\la_1,\ldots,\la_{t_1}\in\FF$. The required is proven.
\end{proof}

%------L3.8--------------------------------------------------
\begin{lemma}\label{lemma_GL_sets}
Given pairwise different letters $x_0,x,e_1,e_2,\ldots$, consider an endomorphism  $\varphi$ of $\X$ defined by
$$\varphi(a)=\left\{
\begin{array}{rl}
x_0^i x,& \text{ if }a=e_i\\
a,& \text{ otherwise }\\
\end{array}
\right.$$
for any letter $a$. Let $\Theta_{I}\subset\EX$ ($\Theta_{II}\subset\EX$, respectively) be the set of all monomials in $x_0,x$ (in $x,e_1,e_2,\ldots$, respectively). Then $\varphi$ induces the well-defined bijection $\ov{\varphi}:\ov{\Theta}_{II} \sqcup\{\ov{x_0}\}\to \ov{\Theta}_I$ of sets of $\stackrel{c}{\sim}$-equivalence classes.
\end{lemma}
\begin{proof}
For $a\in\Theta_{II}$ we have $a=x$ or $a\stackrel{c}{\sim} e_{i_1}x^{j_1}\cdots e_{i_s}x^{j_s}$ for some $i_1,\ldots,i_s>0$ and $j_1,\ldots,j_s\geq0$. Then $\ov{\varphi}(\ov{a})=\ov{b}$ for $b=x$ or $b\stackrel{c}{\sim} x_0^{i_1}x^{j_1+1}\cdots x_0^{i_s}x^{j_s+1}$, respectively. The following fact completes the proof: if $c_1,c_2\in\X$, $c_1$ is primitive, and $c_1\stackrel{c}{\sim} c_2$, then $c_2$ is also primitive. 
\end{proof}

%------L3.9--------------------------------------------------
\begin{lemma}\label{lemma_GL_key}
If $\un{t}\in\NN^u$ with $u>1$, then $\si_{\un{t}}(x_1,\ldots,x_u)\in J_{|\un{t}|-t_1}$. 
\end{lemma}
\begin{proof}
We work in $\AlgLarge$. Assume that $u=2$. It is convenient to denote $x_0=x_1$, $x=x_2$, $e_1=x_3$, $e_2=x_4$ and so on. For short, we set $\un{t}=(k,t)$.  In what follows, we use notations from Lemma~\ref{lemma_GL_sets}. Let $\Upsilon_I$ be the set of finite multisubsets of $\Theta_I$ and $\Upsilon_{II}$ be the set of finite multisubsets of $\Theta_{II} \sqcup\{x_0\}$. We define the $\stackrel{c}{\sim}$-equivalence on $\Upsilon_I$ naturally and denote by $\ov{\Upsilon}_I$ the set of all $\stackrel{c}{\sim}$-equivalence classes. Similarly we define $\ov{\Upsilon}_{II}$. Then Lemma~\ref{lemma_GL_sets} implies that $\ov{\varphi}:\ov{\Upsilon}_{II}\to \ov{\Upsilon}_I$ is a bijection. 

Let us recall that $\Omega(\un{t})$ was defined in Section~\ref{section2}. Assume that $\omega$ belongs to $\ov{\Upsilon}_{I}$ or $\ov{\Upsilon}_{II}$. Since we work in $\AlgLarge$, the element $\si(\omega)$ is well-defined. For short, we write $\mdeg(\omega)$ for $\mdeg(\si(\omega))$. 
By the definition, 
\begin{eq}\label{eq1_lemma_GL_key}
 \si_{(k,t)}(x_0,x)=(-1)^{k+t} \sum_{\omega\in \ov{\Omega}_I} \si(\omega),
\end{eq}%
where $\ov{\Omega}_I=\ov{\Omega}(k,t)=\{\omega\in\ov{\Upsilon}_{I}\,|\,\mdeg(\omega)=(k,t)\}$. For $\ov{\Omega}_{II}=\{\omega\in\ov{\Upsilon}_{II}\,|\,\mdeg(\ov{\varphi}(\omega))=(k,t)\}$ an isomorphism of sets $\ov{\Omega}_{II}\simeq \ov{\Omega}_I$ is determined by the restriction of $\ov{\varphi}$. 

Given $d\geq0$ and $\Delta=(\al_0,\al,\al_1,\ldots,\al_d)\in\NN_0^{d+2}$, where $\al_d>0$ in case $d>0$, we denote $\varphi(\Delta)=\mdeg(\varphi(x_0^{\al_0}x^{\al}e_1^{\al_1}\cdots e_d^{\al_d}))=(\al_0+\sum_{i=1}^d i\al_i,\al+\sum_{i=1}^d \al_i)$ and $\ov{\Omega}_{II}^{\Delta}=\{\omega\in\ov{\Upsilon}_{II}\,|\,\mdeg(\omega)=\Delta\}$. Thus 
\begin{eq}\label{eq2_lemma_GL_key}
 \ov{\Omega}_{II}=\bigsqcup \ov{\Omega}_{II}^{\Delta}, 
\end{eq}%
where the union ranges over $\Delta$ satisfying $\varphi(\Delta)=(k,t)$. Consequently applying formula~(\ref{eq1_lemma_GL_key}), the isomorphism $\ov{\Omega}_{II}\simeq \ov{\Omega}_I$, and formula~(\ref{eq2_lemma_GL_key}) we obtain
$$\si_{(k,t)}(x_0,x)=(-1)^{k+t}\!\!\!\! \sum_{\varphi(\Delta)=(k,t)}\; \sum_{\omega\in\ov{\Omega}_{II}^{\Delta} } \si(\ov{\varphi}(\omega)).$$%
Note that $\sum_{\omega\in\ov{\Omega}_{II}^{\Delta} } \si(\ov{\varphi}(\omega))=(-1)^{\al_0+|\De'|} \si_{\al_0}(x_0)\,  \si_{\Delta'}(x,\varphi(e_1),\ldots,\varphi(e_d))$, where $\Delta'$ stands for $(\al,\al_1,\ldots,\al_d)$. Since the condition $\varphi(\Delta)=(k,t)$ implies $|\De'|=t$, we have 
\begin{eq}\label{eq3_lemma_GL_key}
 \si_{(k,t)}(x_0,x)=\sum (-1)^{\al_0+k}\si_{\al_0}(x_0)\, \si_{\Delta'}(x,x_0 x,x_0^2 x,\ldots,x_0^d x)
\end{eq}%
for $\Delta'=(\al,\al_1,\ldots,\al_d)$, where the sum ranges over $d\geq0$, $\al_0,\al,\al_1,\ldots,\al_d\geq0$ such that $\al_d>0$ in case $d>0$, 
$$\al_0+\sum_{i=1}^d i\al_i = k \;\text{ and }\;\al + \sum_{i=1}^d \al_i = t.$$
Thus, $\si_{(k,t)}(x_0,x)\in J_t$ and the required is proven for $u=2$.

Let $u>2$. By the considered case of the lemma, $\si_{(t_1,t-t_1)}(x_1,x_2)\in J_{t-t_1}$ for $t=|\un{t}|$. Remark~\ref{remark_GL_OK1} implies that $\si_{\un{t}}(x_1,\ldots,x_u)$ is a partial linearization of $\si_{(t_1,t-t_1)}(x_1,x_2)$. The fact that the ideal $J_{t-t_1}$ is closed with respect to taking partial linearizations completes the proof. 
\end{proof}

%------Ex3.10--------------------------------------------------
\begin{example}\label{ex_GL_1}
For letters $x_0$ and $x$ the following equalities hold in $\AlgLarge$: 
\begin{enumerate}
\item[$\bullet$] $\si_{(1,1)}(x_0,x)= \tr(x_0)\tr(x) - \tr(x_0x)\in J_1$; 

\item[$\bullet$] $\si_{(2,2)}(x_0,x) = \si_2(x_0)\si_2(x) - \tr(x_0)\si_{(1,1)}(x,x_0x) + \si_2(x_0x) + \si_{(1,1)}(x,x_0^2x) \in J_2$; 

\item[$\bullet$] for $\un{t}=(1,t_2,\ldots,t_u)$ we have $\si_{\un{t}}(x_1,\ldots,x_u) = \tr(x_1) \si_{(t_2,\ldots,t_u)}(x_2,\ldots,x_u) - \sum_{i=2}^u \si_{\un{t}^{(i)}}(x_1x_i, x_2,\ldots,x_u)\in J_{|\un{t}|-1}$, where $\un{t}^{(i)}=(1,t_2,\ldots,t_i-1,\ldots,t_u)$.
\end{enumerate}
The first two equalities are partial cases of key formula~(\ref{eq3_lemma_GL_key}) from the proof of Lemma~\ref{lemma_GL_key}.
\end{example}
\bigskip

A statement similar to the following lemma was proved by Samoilov in~\cite{Samoilov_2008}.

%------L3.11--------------------------------------------------
\begin{lemma}\label{lemma_GL_key2}
For every $\un{t}\in\NN^u$ we have $\si_{\un{t}}(x_1,\ldots,x_u)\in J_{|\un{t}|}^{(p)}$.
In particular, $\si_t(x_1)\in\AlgLarge$ is a polynomial in $\si_{p^i}(x_1^j)$ for $i\geq0$ and $j\geq1$.  
\end{lemma}
\begin{proof} We work in $\AlgLarge$. If $p=0$, then applying Lemma~\ref{lemma_GL_OK2} several times we obtain the first claim of the lemma.

Assume $p>0$. The positive integer $t_1$ can be written in a base $p$ expansion in the following form: $t_1=\sum_{i=1}^{k} l_i p^{\al_i}$ for $1\leq l_1,\ldots,l_k\leq p-1$, $0\leq \al_1<\cdots<\al_k$, and $k\geq1$. Denote $l=l_1+\cdots+l_k$, 
$$\un{t}' = (\underbrace{p^{\al_1},\ldots,p^{\al_1}}_{l_1},\ldots, \underbrace{p^{\al_k},\ldots,p^{\al_k}}_{l_k},t_2,\ldots,t_u),
\text{ and }
x_1^{(i)}=(\underbrace{x_1,\ldots,x_1}_i).
$$%
Lemma~\ref{lemma_GL_OK2} implies that for $\al=(p^{\al_1}!)^{l_1}\cdots (p^{\al_k}!)^{l_k}$ we have
$$\al\, \si_{\un{t}'}(x_1^{(l)},x_2,\ldots,x_u)=\si_{(1^{t_1},t_2,\ldots,t_u)}(x_1^{(t_1)},x_2,\ldots,x_u)\text{ and }$$
$$t_1!\, \si_{\un{t}}(x_1,\ldots,x_u)=\si_{(1^{t_1},t_2,\ldots,t_u)}(x_1^{(t_1)},x_2,\ldots,x_u).$$
Hence 
\begin{eq}\label{eq4}
 \si_{\un{t}}(x_1,\ldots,x_u)=\be_{t_1} \si_{\un{t}'}(x_1^{(l)},x_2,\ldots,x_u)
\end{eq}%
over $\QQ$, where $\be_{t_1}=\al/t_1!$. We claim that 
\begin{eq}\label{eq_statement_beta}
\be_{t_1}\neq0 \text{ is well-defined over an arbitrary field } \FF \text{ of characteristic } p. 
\end{eq}%
Denote $[\be]=\max\{\ga\in\ZZ\,|\,\ga\leq \be\}$ for $\be\in\QQ$. Given $m=p^\ga q\in\NN$ with $\ga,q\in\NN_0$ such that $p$ is not a divisor of $q$, we write $\ga_m$ for $\ga$. Note that $\ga_{m!}=\sum_{j\geq 1} \#\{1\leq i\leq m\,|\,p^j \text{ is a divisor of } i\}=\sum_{j\geq 1}[m/p^j]$. Therefore,
$$\ga_{t_1!}=\sum_{j\geq 1}\sum_{i=1}^k l_i\, [p^{\al_i-j}]=\sum_{i=1}^k l_i \left(p^{\al_i-1}+p^{\al_i-2}+\cdots+1\right)=\ga_{\al}.$$
Statement~(\ref{eq_statement_beta}) is proven. Thus~(\ref{eq4}) holds over $\FF$. Repeating this procedure for $t_2,\ldots,t_u$, we obtain the first claim of the lemma.

The proven part of the lemma implies that $\si_t(x_1)$ belongs to $\FF$-span of $\si_{\un{t}}(x_1,\ldots,x_1)$, where $\un{t}\in\NN^u$, $|\un{t}|=t$,  and $t_i\in\{1,p,p^2,\ldots\}$ for all $i$. If $u>1$, then $\si_{\un{t}}(x_1,\ldots,x_1)$ is a polynomial in $\si_k(x_1^j)$ for $1\leq k<t$ and $j>0$. So we can apply the above reasoning to $\si_k(x_1^j)$ and so on. Finally, we prove the second claim of the lemma.
\end{proof}

Now we can prove Theorem~\ref{theo_GL} and Remark~\ref{remark_theo_GL}:
\begin{proof}
We work in $\AlgLarge$. Since the field $\FF$ is infinite, part~2 of Theorem~\ref{theo_Zubkov} together with Remark~\ref{remark_GL_OK1} implies that $\KLarge{n}$ is generated by 
\begin{enumerate}
\item[(a)] $\si_{\un{t}}(a_1,\ldots,a_u)=0$ for $|\un{t}|>n$, where $\un{t}\in\NN^u$, $u>1$, and $a_1,\ldots,a_u\in\X$;

\item[(b)] $\si_t(b)=0$ for $t>n$, where $b\in\X$.
\end{enumerate}

Consider relations~(a). For short, we set $t=|\un{t}|$ and $\un{a}=(a_1,\ldots,a_u)$. Since $\si_{\un{t}}(\un{a})=\si_{\un{t}^{\si}}(a_{\si(1)},\ldots,a_{\si(u)})$ for all $\si \in S_u$, where $\un{t}^{\si}$ stands for $(t_{\si(1)},\ldots,t_{\si(u)})$, we can always assume that $t_1\geq \cdots \geq t_u$. By Lemma~\ref{lemma_GL_key2} we can assume that $t_i\in\{1,p,p^2,\ldots\}$ for all $i$. If $t>n+1$ and $t_i=1$ for some $i$, then Lemma~\ref{lemma_GL_key} implies that $\si_{\un{t}}(\un{a})\in J_{t-1}$. Repeating this procedure several times we obtain that every relation from~(a) follows from relations~(b) and such relations~(a) that have $\un{t}$ satisfying one of the following conditions:  
\begin{enumerate}
\item[1)] $t=n+1$ and $t_i\in\{1,p,p^2,\ldots\}$ for all $i$;

\item[2)] $t>n+1$ and $t_i\in\{p,p^2,\ldots\}$ for all $i$. 
\end{enumerate}

Consider the second case. If $t-t_{u}>n$, then $\si_{\un{t}}(\un{a})\in J_{t-t_u} \subset \KLarge{n}$ by Lemma~\ref{lemma_GL_key} and  part~2 of Theorem~\ref{theo_Zubkov}. Therefore, we can assume that $t-t_{u}\leq n$. If $p>n$, then $t-t_u=t_1+\cdots+t_{u-1}\geq p>n$; a contradiction. Thus in case $p>n$ we can assume that $\un{t}$ satisfies condition~1). If $p\leq n$ and $t> 2n$, then $n\geq t-t_u> 2n - t_u$; thus $t_1,\ldots,t_u>n$ and $t-t_u > n$; a contradiction.  Thus in case $p\leq n$ we can assume that $t\leq 2n$. Hence, $\un{t}$ satisfies conditions~(\ref{eq_cond1}), (\ref{eq_cond2}), (\ref{eq_cond3}). 

Consider relations~(b). If $b=c^l$ for $l>1$ and $c\in\EX$, then $\si_t(b)=P_{t,l}(c)$ in $\AlgLarge$ (see Lemma~\ref{lemma_GL_Donkin}). Part~1 of Lemma~\ref{lemma_P} implies that relations (b) follow from 
\begin{enumerate}
\item[($\rm b'$)] $\si_t(b)=0$ for $t>n$, where $b\in\EX$.
\end{enumerate}
The required is proven.
\end{proof}

As we have already mentioned, the first part of the following lemma is a reformulation of the result from~\cite{Razmyslov_1974} and~\cite{Procesi_1976}. The second part is new.

%------L3.12--------------------------------------------------
\begin{lemma}\label{lemma_GL_fb}
If $p=0$, then the ideal $\KLarge{n}\vartriangleleft \AlgLarge$ is generated by $\si_{n+1}(a)=0$ for $a\in\FF\X$; in particular, $\KLarge{n}$ is finitely based.

If $p>0$, then the ideal $\KLarge{n}\vartriangleleft \AlgLarge$ is not finitely based. 
\end{lemma}
\begin{proof}
We work in $\AlgLarge$. Assume $p=0$. For short, we write $x$ for $x_1$ and $I$ for the ideal generated by $\si_{n+1}(a)=0$ for all $a\in\FF\X$. Theorem~\ref{theo_GL} together with Remarks~\ref{remark_theo_GL} and~\ref{remark_GL_OK1} implies that $\KLarge{n}$ lies in the ideal, generated by $I$ and $\si_t(b)=0$ for $t>n$ and $b\in\EX$. By Lemma~\ref{lemma_GL_key2}, $\si_t(x)$ is equal to a polynomial $f_t$ in $\tr(x^j)$, where $1\leq j\leq t$. Given $t>n$, we make substitutions $\tr(x^j)\to-\sum_{i=1}^n (-1)^{i} f_i \tr(x^{j-i})$, where $j>n$, in $f_t$. Repeating this procedure several times we obtain a polynomial in $\tr(x),\ldots,\tr(x^n)$, which we denote by $h_t$.  Since $\si_{(n,1)}(x,x^{j-n})=\sum_{i=0}^n (-1)^{n-i} \si_{i}(x) \tr(x^{j-i})$ belongs to $I$ for $j>n$, the element $h_t$ is equal to $f_t$ modulo the ideal $I$. If $h_t$ is not equal to zero in $\AlgLarge$ for $t>n$, then the equality $\PhiLarge{n}(h_t)=0$ implies that $\tr(X_1),\ldots,\tr(X_1^n)$ are not algebraically independent over $\FF$; a contradiction. Thus, for $t>n$ we have $h_t=0$ and $f_t\in I$. The required is proven. 

Assume $p>0$. Let $\KLarge{n}$ be finitely based, i.e., $\KLarge{n}$ is generated by some elements $f_1,\ldots,f_r\in\AlgLarge$ as $\Tid$-ideal. Denote by $m$ the maximal $t>0$ such that $\si_t(a)$ is a multiple of a summand of $f_i$ for some $i$ and $a\in\EX$. Consider a letter $x$ and $k\in\NN$ satisfying $p^k>m$ and $p^k>n$. We claim that 
\begin{eq}\label{eq_claim}
\!\!\!\!\!\!\!\!\si_{p^k}(x)\in \KLarge{n} \text{ does not belong to }\Tid\text-\text{ideal of } \AlgLarge \text{ generated by } f_1,\ldots,f_r. 
\end{eq}%
If the claim does not hold, then 
\begin{eq}\label{eq2}
\si_{p^k}(x)=\sum_{i=1}^r f_i(\un{a}_i) f'_i
\end{eq}% 
in $\AlgLarge$,  where $f'_i\in\AlgLarge$, $\un{a}_i=(a_{i1},\ldots,a_{ir_i})$ for $a_{ij}\in\FF\X$ for all $i,j$, and $f_i(\un{a}_i)$ stands for the result of substitutions $x_1\to a_{i1},\ldots, x_{r_i}\to a_{ir_i}$ in $f_i$. The right hand side of~(\ref{eq2}) is equal in $\AlgLarge$ to a polynomial in elements $\{\si_i(b)\,|\, i>0,\; b\in\EX\}$, which are algebraically independent in $\AlgLarge$ (see Lemma~\ref{lemma_GL_Donkin}). Thus, there exists an $i$ such that the polynomial $f_i(\un{a}_i)\in\AlgLarge$ contains a summand $\si_{p^k}(x)$. Therefore, there are $t<p^k$, $c_1,\ldots,c_s\in\X$, and $\ga_1,\ldots,\ga_s\in\FF$ such that $\si_t(\ga_1 c_1+\cdots+\ga_s c_s)\in\AlgLarge$ contains a summand $\si_{p^k}(x)$.  To write down $\si_t(\ga_1 c_1+\cdots+\ga_s c_s)$ as a polynomial in $\{\si_i(b)\,|\, i>0,\;b\in\EX\}$ we apply relations~(a) from Lemma~\ref{lemma_GL_Donkin} and then apply relations~(b) from Lemma~\ref{lemma_GL_Donkin} to the result several times. Note that the representation of $\si_t(\ga_1 c_1+\cdots+\ga_s c_s)$ in the mentioned form is unique (see Lemma~\ref{lemma_GL_Donkin}).  Since $F_{t}(\ga_1 c_1,\ldots,\ga_s c_s)\in\si\X$ does not contain a summand $\si_{p^k}(x)$, we conclude that there exist $j<p^k$, $c\in\X$, and $l>1$ such that $\si_j(c^l)=P_{j,l}(c)\in\AlgLarge$ contains a summand $\si_{p^k}(x)$.  

It is not difficult to see that we can assume that $c=x$ and $jl=p^k$. Hence $j=p^{k_0}$ and $l=p^{k-k_0}$ for $0\leq k_0<k$. By part~2 of Lemma~\ref{lemma_P}, we obtain $P_{j,l}(c)=\si_j(x)^l\in\AlgLarge$ does not contain a summand $\si_{p^k}(x)$; a contradiction.
Thus, claim~(\ref{eq_claim}) does hold and $\KLarge{n}$ is not finitely based.
\end{proof}

%=====================================================================
%=====================================================================
%------Sec4--------------------------------------------------
\section{Proof of Theorem~\ref{theo_GL_main}}\label{section4}

We prove Theorem~\ref{theo_GL_main} together with Remark~\ref{remark_theo_GL_main} at the end of this section.  In this section we assume that~(a), (b), (c) are relations from part~2 of Theorem~\ref{theo_GL_main}. Similarly,~(d) stands for those relations (d) from part~2 of Theorem~\ref{theo_GL_main} that satisfy conditions from Remark~\ref{remark_theo_GL_main}.
As above, we assume that $n>1$.

%------L4.1--------------------------------------------------
\begin{lemma}\label{lemma_GL_3}
The ideal $\KSmall{n}$ of relations for $R^{GL(n)}\simeq \AlgSmall{n} / \KSmall{n}$ is generated by relations~(a), (c), (d) and 
\begin{enumerate}
 \item[($b_{t,l}$)] $\si_t(a^l)=P_{t,l}^{\pplus}(a)$ for $1\leq t\leq n$, $l>1$, $a\in\X$. 
\end{enumerate}
\end{lemma}% 
\begin{proof}
Let $I$ be the ideal of $\AlgSmall{n}$ generated by the elements from the formulation of the lemma. Note that in Lemma~\ref{lemma_GL_Donkin} we can assume that relations~(a) satisfy $u=2$ in case $t\leq n$ and $a_i=\be_i b_i$ for $\be_i\in\FF$, $b_i\in\X$ in case $t>n$. Since the ideal $\KLarge{n}$ is described by part~2 of Theorem~\ref{theo_Zubkov} and $\Ker(\piLarge)=L$ is considered in Lemma~\ref{lemma_GL_Donkin}, we obtain that the ideal 
$\Ker(\PhiAbs{n})=\piLarge^{-1}(\KLarge{n})$ is generated by
\begin{enumerate}
\item[(${\rm a}'$)] $\si_t(a+b)=F_t(a,b)$, where $1\leq t\leq n$, $a,b\in\FF\X$; 

\item[(${\rm a}''$)] $\si_t(\al_1 a_1+\cdots+\al_u a_u)=F_t(\al_1 a_1,\ldots,\al_u a_u)$, where $u>1$, $t>n$, $\al_i\in\FF$, $a_i\in\X$ for all $i$; 

\item[(${\rm b}'$)] $\si_t(a^l)=P_{t,l}(a)$, where $t>0$, $l>1$, $a\in\X$;

\item[(${\rm c}'$)] $\si_t(ab)=\si_t(ba)$, where $t>0$, $a,b\in\X$;

\item[(${\rm d}'$)] $\si_t(a)=0$, where $t>n$, $a\in\FF\X$. 
\end{enumerate}
We have $\piSmall{n}(\Ker(\PhiAbs{n}))=\KSmall{n}$. Since elements $\si_{\un{t}}(a_1,\ldots,a_u)$, where $|\un{t}|>n$, $a_i\in\X$, belong to $\Ker(\PhiAbs{n})$, relations~(d) belong to $\KSmall{n}$.  Note that $\piSmall{n}$ sends (${\rm d}'$) to zero. Hence the ideal $\KSmall{n}$ of $\AlgSmall{n}$ is generated by the ideal $I$ and the images of (${\rm a}''$), (${\rm b}'$) in $\AlgSmall{n}$. 
By part~1 of Lemma~\ref{lemma_P}, in case $t>n$ the image of relations (${\rm b}'$) in $\AlgSmall{n}$ is equal to zero. Since the field $\FF$ is infinite, we can take elements 
\begin{enumerate}
\item[(${\rm a}'''$)] $\si_{\un{t}}(a_1,\ldots,a_u)\in \AlgSmall{n}$, where $u>1$, $|\un{t}|>n$, $a_1,\ldots,a_u\in\X$,  
\end{enumerate}
instead of the image of (${\rm a}''$) in $\AlgSmall{n}$. By Theorem~\ref{theo_GL} and Remark~\ref{remark_theo_GL}, an element 
$\si_{\un{t}}(a_1,\ldots,a_u)$, where $|\un{t}|>n$, $a_i\in\X$, of $\si\X$ belongs to the ideal of $\si\X$, generated by relations (${\rm b}'$), (${\rm c}'$), (${\rm d}'$), and (d),  considered as elements of $\si\X$. Therefore, relations (${\rm a}'''$) belong to $I$. The required is proven.
\end{proof}

To complete the proof of part~2 of Theorem~\ref{theo_GL_main} it is enough to show that in Lemma~\ref{lemma_GL_3} we can assume that $1< l\leq n$ in relations ($b_{t,l}$). We prove this fact in Lemma~\ref{lemma_GL_2} (see below). The definition of $\equiv$-equivalence was given in Section~\ref{section2}.

%------L4.2--------------------------------------------------
\begin{lemma}\label{lemma_GL_equiv}
Given letters $x,y$ and $1\leq t\leq n$, we have that $\si_t(x^n y)\equiv0$ in $\AlgSmall{n}$ follows from $\si_{(kn,k)}(a,b)=0$ and (c), where $1\leq k\leq t$ and $a,b\in\X$.
\end{lemma}
\begin{proof} We work in the quotient of $\AlgSmall{n}$ by the ideal generated by (c). Assume that $x\neq y$. The proof is by induction on $1\leq t\leq n$. 

Let $t=1$. Since 
$$(-1)^{n}\si_{(n,1)}(x,y)=\tr(x^ny) + \sum_{i=1}^n (-1)^i \tr(x^{n-i}y)\si_i(x),$$ 
we obtain the required. 

Assume $t>1$. Denote by $J$ the ideal generated by $\si_{(kn,k)}(a,b)=0$, $1\leq k\leq t$, $a,b\in\X$. Let $\Theta_t$ be the set of monomials $c\in\EX$ in letters $x,y$ such that $\deg_x(c)=nt$ and $\deg_y(c)=t$. Since the relation $\si_{(tn,t)}(x,y)=0$ belongs to $J$, we have that
\begin{eq}\label{eq3} 
\sum (-1)^{t(n+1)-r} \!\!\!\!\sum_{c\in \Theta_{t/r}} \si_r(c)\equiv0 \;\text{ holds modulo }J, 
\end{eq}%
where the sum is taken over $1\leq r\leq t$ with $r|t$. Note that for every $c\in\Theta_{t/r}$ there exists a $c_0\in\X$ such that $c\stackrel{c}{\sim} x^n c_0$. Thus the induction hypothesis implies that for $1\leq r<t$ we have that $\si_r(c)\equiv0$ follows from $\si_{(kn,k)}(a,b)=0$, where $1\leq k\leq r<t$ and $a,b\in\X$. Thus, $\si_r(c)\equiv0$ holds modulo $J$. Since $\Theta_1=\{c\}$ for $c\stackrel{c}{\sim}x^n y$, formula~(\ref{eq3}) implies the required. 
\end{proof}

Let $I_n$ be the ideal of $\AlgSmall{n}$ generated by relations~(a), (b), (c), (d).

%------Rem4.3--------------------------------------------------
\begin{remark}\label{rem1} The ideal $I_n$ is closed with respect to substitutions $x_i\to a_i$ for $i>0$ and $a_i\in\X$. 
\end{remark}

%------L4.4--------------------------------------------------
\begin{lemma}\label{lemma_GL_2}
Relations $(b_{t,l})$ belong to $I_n$, where $l>1$ and $1\leq t\leq n$.  
\end{lemma}
\begin{proof}
We work in the quotient of $\AlgSmall{n}$ by the ideal generated by (c). Assume that $x,y$ are different letters. We prove by induction on $r>1$ the claim that relations
$$\begin{array}{cl}
{\rm (b_{\it t,l})}\;\, & \text{for }tl=r,\\
{\rm (h_{\it k})}:& \si_{(kn,n)}(x,y)=0  \text{ for }k(n+1)=r
\end{array}
$$
belong to $I_n$ for all $1\leq t\leq n$, $l>1$, $k\geq1$. 

Let $r=2$. Since $n\geq2$, we obtain that $\rm (b_{1,2})$ belongs to $I_n$ and the set of relations ${\rm (h_{\it k})}$ is empty.

Assume that $r>2$. Let $h$ be a relation from $\rm (h_{\it k})$, where $k(n+1)=r$. Note that we can not claim that $h$ is a relation from (d). Since $h$ lies in $\KSmall{n}$, Lemma~\ref{lemma_GL_3} implies that $h$ belongs to $I_n$ modulo some relations from $\rm (b_{\it i,j})$ for $i,j>0$ satisfying $ij\leq \max\{\deg_x(h),\deg_y(h)\}=kn<r$. By the induction hypothesis, the mentioned relations from $\rm (b_{\it i,j})$ belong to $I_n$. Thus, all relations from $\rm (h_{\it k})$ belong to $I_n$.  

Consider $\rm (b_{\it t,l})$, where $tl=r$. If $l\leq n$, then relations $\rm (b_{\it t,l})$ belong to $I_n$ by the definition. 

Let $l>n$.  By Lemma~\ref{lemma_GL_equiv}, $\si_t(x^n y)\equiv0$ follows from $\si_{(in,i)}(a,b)=0$,  where $1\leq i\leq t$ and $a,b\in\X$. Note that $i(n+1)\leq tl=r$. If $i(n+1)<r$, then the induction hypothesis implies that $\rm (h_{\it i})$ belongs to $I_n$. On the other hand, if $i(n+1)=r$, then the proven part of the claim implies that $\rm (h_{\it i})$ belongs to $I_n$. By Remark~\ref{rem1}, 
$$\si_t(x^n y)\equiv0 \text{ holds modulo }I_n.$$% 
Since $l>n$, we obtain that modulo the ideal $I_n$ the element $\si_t(x^l)=\si_t(x^n x^{l-n})$ is a polynomial in $\si_i(x^j)$ for $ij<\deg{\si_t(x^l)}=r$ with $1\leq i\leq n$. Applying $(b_{\it i,j})$ to $\si_i(x^j)$ and using the induction hypothesis, we can see that there is a polynomial $P_{t,l}'(x)$ in $\si_1(x),\ldots,\si_n(x)$ such that $\si_t(x^l)=P_{t,l}'(x)$ modulo the ideal $I_n$. Let $X$ be the $n\times n$ generic matrix corresponding to the letter $x$. Since $I_n$ is a subset of $\KSmall{n}$, $P_{t,l}'(X)=\si_t(X^l)=P_{t,l}^{\pplus}(X)$ is the equality of polynomials in $\si_1(X),\ldots,\si_n(X)$. But it is well-known that the latter elements are algebraically independent over $\FF$. Therefore, $P_{t,l}'(x)=P_{t,l}^{\pplus}(x)$ and  $\si_t(x^l)=P_{t,l}(x)$ holds modulo $I_n$. By Remark~\ref{rem1}, relations $\rm (b_{\it t,l})$ belong to $I_n$. The claim is proven. 
\end{proof}

Now we can prove Theorem~\ref{theo_GL_main} and Remark~\ref{remark_theo_GL_main}.

\begin{proof}
By part~1 of Theorem~\ref{theo_Zubkov}, the ideal $\TLarge{n}$ is generated by  $\KLarge{n}\otimes 1$ and $\chi_n(a)=0$ for $a\in\FF\X$. For the sake of completeness, we point out that the mentioned result is a partial case of Lemma~\ref{lemma_73} (see below). Consider the surjective map $\PsiAbs{n}=\PhiAbs{n}\otimes \phi_n:\si\X\otimes\FF\X^{\#} \to \C_n$. Then the following diagram is commutative: 
$$ 
\begin{picture}(0,120)
\put(0,95){%
\put(0,-2){\vector(0,-1){58}}%
\put(15,0){\vector(3,-2){33}}%
\put(-15,0){\vector(-3,-2){33}}%
\put(50,-40){\vector(-3,-2){35}}%
\put(-50,-40){\vector(3,-2){35}}%
\put(110,0){\vector(-3,-2){33}}%
\put(-110,0){\vector(3,-2){33}}%
\put(-29,5){$\si\X\otimes \FF\X^{\#}$}%
\put(-100,-33){$\AlgSmall{n}\otimes \FF\X^{\#}$}%
\put(35,-33){$\AlgLarge\otimes \FF\X^{\#}$}%
\put(-6,-75){$\C_n$}%
\put(-125,5){$\TSmall{n}$}%
\put(115,5){$\TLarge{n}$}%
\put(3,-33){$\scriptstyle\PsiAbs{n}$}% 
\put(-54,-8){$\scriptstyle\piSmall{n}\otimes\idmap$}%
\put(33,-8){$\scriptstyle\piLarge\otimes\idmap$}%
\put(-35,-48){$\scriptstyle\PsiSmall{n}$}%
\put(25,-48){$\scriptstyle\PsiLarge{n}$}%
\put(-20,-90){\text{Diagram 3.}}%
}%
\end{picture}
$$%
Here $\idmap$ stands for the identical map on $\FF\X^{\#}$. The kernel $\TAbs{n}$ of $\PsiAbs{n}$ is equal to $(\piLarge\otimes\idmap)^{-1}(\TLarge{n})$.   
Thus, the ideal $\TAbs{n}$ is generated by $(\piLarge\otimes\idmap)^{-1}(\KLarge{n}\otimes 1)=\piLarge^{-1}(\KLarge{n})\otimes 1 +  \Ker(\piLarge)\otimes \FF\X^{\#}$ and $(\piLarge\otimes\idmap)^{-1}(\chi_n(a))$, $a\in\FF\X$. Since $\piLarge^{-1}(\KLarge{n})=\Ker(\PhiAbs{n})$ and $\Ker(\piLarge\otimes\idmap)=\Ker(\piLarge)\otimes \FF\X^{\#} \subset \Ker(\PhiAbs{n})\otimes \FF\X^{\#}$, the ideal $\TAbs{n}$ is generated by $\Ker(\PhiAbs{n})\otimes 1$ and $\chi_n(a)$, $a\in\FF\X$. 

Note that $\TSmall{n}=(\piSmall{n}\otimes\idmap)(\TAbs{n})$. Thus the ideal $\TSmall{n}$ is generated by $(\piSmall{n}\otimes\idmap)(\Ker(\PhiAbs{n})\otimes 1)=\KSmall{n}\otimes 1$ and $\chi_n(a)$, $a\in\FF\X$, considered as an element of $\AlgSmall{n}\otimes \FF\X^{\#}$. Hence part~1 of Theorem~\ref{theo_GL_main} is proven. Part~2 of Theorem~\ref{theo_GL_main} follows immediately from Lemmas~\ref{lemma_GL_3} and~\ref{lemma_GL_2}.
\end{proof}

%=====================================================================
%=====================================================================
%------Sec5--------------------------------------------------
\section{Relations for matrix $O(n)$-invariants}\label{section5}

In the rest of the paper we assume that $p\neq2$. In this section we consider identities with forms for the $\FF$-algebra generated by $n\times n$ generic and transpose generic matrices or, equivalently, identities for the algebra $\CY_n$. 

The algebra of {\it matrix $O(n)$-invariants} $R^{O(n)}$ is known to be generated by $\si_t(A)$, where $1\leq t\leq n$ and $A$ is a monomial in generic and transpose generic matrices. The mentioned generators of $R^{O(n)}$ were found by Sibirskii~\cite{Sibirskii_1968} and Procesi~\cite{Procesi_1976} in characteristic zero case and by Zubkov~\cite{Zubkov_1999} in the general case. In Section~\ref{section1} we denoted by 
$$\CY_n=\alg_{\FF}\{X_1,X_1^T,X_2,X_2^T,\ldots,fE\}$$ 
the algebra generated by generic matrices, transpose generic matrices and $fE$, where $f$ ranges over $R^{O(n)}$.  

Similarly to Section~\ref{section1} we define the following notions.
\begin{enumerate}
\item[$\bullet$] Let $\Y$ be the semigroup (without unity) freely generated by {\it letters}  $x_1,x_1^T,x_2,x_2^T,\ldots$ and $\Y^{\#}=\Y\sqcup\{1\}$.

\item[$\bullet$] Introduce a lexicographical linear order on $\Y$ by setting $x_1>x_1^T>x_2>x_2^T>\cdots$ and $ab>a$ for $a,b\in\Y$.

\item[$\bullet$] Introduce the involution ${}^T$ on $\Y$ as follows. We set $(x_k)^T=x_k^T$, $(x_k^T)^T=x_k$ for all $k$ and $(a_1\cdots a_s)^T=a_s^T\cdots a_1^T\in\Y$. 

\item[$\bullet$] We say that $a,b\in\Y$ are {\it cyclic equivalent} and write $a\stackrel{c}{\sim} b$
if $a=a_1a_2$ and $b=a_2a_1$ for some $a_1,a_2\in\Y^{\#}$. If $a\stackrel{c}{\sim} b$ or $a\stackrel{c}{\sim} b^T$, then we say that $a$ and $b$ are {\it equivalent} and write $a\sim b$. 
 
\item[$\bullet$] Using $\Y$ instead of $\X$ and $\sim$-equivalence instead of $\stackrel{c}{\sim}$-equivalence, we introduce $\FF\Y$, $\FF\Y^{\#}$, $\AlgSmallY{n}$, $\si\Y$, $\EY$,  $\AlgLargeY$, respectively, similarly to $\FF\X$, $\FF\X^{\#}$, $\AlgSmall{n}$, $\si\X$, $\EX$,  $\AlgLarge$, respectively. Note that $\AlgSmallY{n}\subset \si\Y$.

\item[$\bullet$] The algebra $\si\Y$ ($\AlgLargeY$, $\AlgSmallY{n}$, respectively) is called the {\it absolutely} ({\it large}, {\it small}, respectively) free algebra for $R^{O(n)}$. 

\item[$\bullet$] Let  $\phi'_n:\FF\Y^{\#}\to \alg_{\FF}\{E,X_1,X_1^T,X_2,X_2^T,\ldots\}$ be the homomorphism of algebras defined by $1\to E$ and $x_k\to X_k$, $x_k^T\to X_k^T$ for all $k\geq1$.
\end{enumerate}%

\noindent{}Given $\un{t}\in\NN^u$ and  $\un{a}=(a_1,\ldots,a_u)$ for $a_1,\ldots,a_u\in\FF\Y$, we define $\si_{\un{t}}(\un{a})\in\si\Y$ as the result of substitutions $x_1\to a_1,\ldots,x_u\to a_u$ in $\si_{\un{t}}(x_1,\ldots,x_u)\in\si\X$. Similarly we define elements $F_t(\un{a})$ and $P_{t,l}(b)$ of $\si\Y$, where $b\in\FF\Y$.

By Lemma~\ref{lemma_O_Lopatin} (see below), we have the surjective homomorphism $\piLargeY:\si\Y\to\AlgLargeY$. Using $\phi'_n$ instead of $\phi_n$, we define the surjective homomorphisms $\PhiAbsY{n}$, $\PhiLargeY{n}$, $\PhiSmallY{n}$, $\piSmallY{n}$, respectively, in the same way as $\PhiAbs{n}$, $\PhiLarge{n}$, $\PhiSmall{n}$, $\piSmall{n}$ (see the diagram below for the details). Denote by $\KLargeY{n}$ and $\KSmallY{n}$ the kernels of $\PhiLargeY{n}$ and $\PhiSmallY{n}$, respectively. Then the following diagram is commutative. Namely, its left triangle is commutative by the definition and the commutability of its right triangle can be shown in the same way as in Remark~\ref{remark_GL_diagram}. 
$$ 
\begin{picture}(0,120)
\put(0,95){%
\put(0,-2){\vector(0,-1){58}}%
\put(15,0){\vector(3,-2){35}}%
\put(-15,0){\vector(-3,-2){35}}%
\put(-11,5){$\si\Y$}%
\put(-75,-33){$\AlgSmallY{n}$}%
\put(50,-33){$\AlgLargeY$}%
\put(50,-40){\vector(-3,-2){35}}%
\put(-50,-40){\vector(3,-2){35}}%
\put(110,0){\vector(-3,-2){35}}%
\put(-110,0){\vector(3,-2){35}}%
\put(-6,-75){$R^{O(n)}$}%
\put(-125,5){$\KSmallY{n}$}%
\put(115,5){$\KLargeY{n}$}%
\put(3,-33){$\scriptstyle\PhiAbsY{n}$}%
\put(-40,-8){$\scriptstyle\piSmallY{n}$}%
\put(33,-8){$\scriptstyle\piLargeY$}%
\put(-35,-48){$\scriptstyle\PhiSmallY{n}$}%
\put(25,-48){$\scriptstyle\PhiLargeY{n}$}%
\put(-20,-90){\text{Diagram 4.}}%
}%
\end{picture}
$$%
\noindent{}Denote the kernels of surjective homomorphisms 
$$\PhiLargeY{n}\otimes\phi'_n:\AlgLargeY \otimes \FF\Y^{\#} \to \CY_n\;\text{ and }\;
\PhiSmallY{n}\otimes\phi'_n:\AlgSmallY{n}\otimes \FF\Y^{\#} \to \CY_n$$
by $\TLargeY{n}$ and $\TSmallY{n}$, respectively. These ideals are ideals of relations for $\CY_n$ in the corresponding free algebras. 

Let $\un{t}\in\NN_0^u$, $\un{r}\in\NN_0^v$, $\un{s}\in\NN_0^w$ ($u,v,w>0$) with $|\un{r}|=|\un{s}|$. We set $y_1=x_{u+1},\ldots,y_v=x_{u+v}$ and $z_1=x_{u+v+1},\ldots,z_w=x_{u+v+w}$. In order to define $\si_{\un{t};\un{r};\un{s}}(\un{x};\un{y};\un{z})\in \FF\Y$ for $\un{x}=(x_1,\ldots,x_u)$, $\un{y}=(y_1,\ldots,y_v)$, $\un{z}=(z_1,\ldots,z_w)$,  we consider the quiver (i.e., the oriented graph) $\Q=\Q(\un{x};\un{y};\un{z})$:
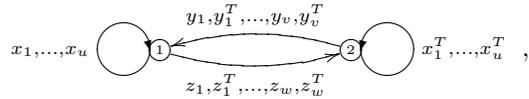
$$
\loopR{0}{0}{x_1,\ldots,x_u} %
\xymatrix@C=1cm@R=1cm{ %
\vtx{1}\ar@1@/^/@{<-}[rr]^{y_1,y_1^T,\ldots,y_v,y_v^T} &&\vtx{2}\ar@1@/^/@{<-}[ll]^{z_1,z_1^T,\ldots,z_w,z_w^T}\\
}%
\loopL{0}{0}{x_1^T,\ldots,x_u^T}\qquad\qquad,
$$
where there are $2v$ ($2w$, respectively) arrows from vertex $2$ to vertex $1$ (from $1$ to $2$, respectively) and there are $u$ loops in each of two vertices. By abuse of notation arrows of $\Q$ are denoted by letters from $\Y$. For an arrow $a$ denote by $a'$ its head and by $a''$ its tail. A sequence of arrows $a_1\cdots a_s$ of $\Q$ is a {\it path} of $\Q$ if $a_i''=a_{i+1}'$ for all $1\leq i< s$. The head of the path $a$ is $a'=a_1'$ and the tail is $a''=a_s''$. A path $a$ is {\it closed} if $a'=a''$. We introduce the following notations:
\begin{enumerate}
\item[$\bullet$] $\path{\Q}$ is the set of all (non-empty) paths in $\Q$;

\item[$\bullet$] $\LA\Q\RA\subset \Y$ is the semigroup (without unity) freely generated by closed paths in $\Q$ and $\LA\Q\RA^{\#}=\LA\Q\RA \sqcup\{1\}$; note that $\LA\Q\RA$ is closed with respect to the $\sim$-equivalence;

\item[$\bullet$] $\LA\widetilde{\Q}\RA = \EY \cap \LA\Q\RA$ is the set of maximal representatives of $\sim$-equivalence classes of primitive elements from $\LA\Q\RA$;

\item[$\bullet$] $\LA\ov{\Q}\RA$ is the set of all $\sim$-equivalence classes of primitive elements from $\LA\Q\RA$.    
\end{enumerate}%
Denote the multidegree of a monomial $a$ in arrows of $\Q$ by $\mdeg(a)=(\deg_{x_1}(a) + \deg_{x_1^T}(a),\ldots,\deg_{z_w}(a) + \deg_{z_w^T}(a))$.

%------Rem5.1--------------------------------------------------
\begin{remark}\label{remark_primitive}
If $a,b\in\LA\Q\RA$ and $a\sim b^l$ for $l>1$, then there exists a $c\in\LA\Q\RA$ such that $a=c^l$.
\end{remark}
\bigskip

Let $\Omega=\Omega(\un{t};\un{r};\un{s})$ be the set of multisets  
$$\omega=\{\underbrace{e_1,\ldots,e_1}_{k_1},\ldots,\underbrace{e_q,\ldots,e_q}_{k_q}\}_{m}
$$%
such that
\begin{enumerate}
\item[$\bullet$] $e_1,\ldots,e_q \in\LA\widetilde{\Q}\RA$ are pairwise different and $k_1,\ldots,k_q\in\NN$ ($q>0$); 

\item[$\bullet$] $k_1\mdeg(e_1)+\cdots+k_q\mdeg(e_q)=(\un{t},\un{r},\un{s})$.
\end{enumerate}
We set $\si(\omega)=(-1)^{\xi} \si_{k_1}(e_1)\cdots\si_{k_q}(e_q)$ for 
$$\xi=\sum_{i=1}^q k_i\left(\sum_{j=1}^v\deg_{y_j}{e_i}+\sum_{j=1}^w\deg_{z_j}{e_i}+1\right).$$  
Then we define the following element of $\FF\Y$:
\begin{eq}
\si_{\un{t};\un{r};\un{s}}(\un{x};\un{y};\un{z})= (-1)^{|\un{t}|}\!\!\! \sum_{\omega\in \Omega(\un{t};\un{r};\un{s})}  \si(\omega).
\end{eq}%
For empty $\Omega$ we set $\si_{\un{t};\un{r};\un{s}}(\un{x};\un{y};\un{z})=1$.
Given $a_i,b_j,c_k\in\FF\Y$, we define $\si_{\un{t};\un{r};\un{s}}(\un{a};\un{b};\un{c})$ as the result of the corresponding substitutions in $\si_{\un{t};\un{r};\un{s}}(\un{x};\un{y};\un{z})$. Note that 
\begin{enumerate}
 \item[$\bullet$] for $\un{r}=\un{s}=(0)$ we have $\si_{\un{t};\un{r};\un{s}}(\un{a};\un{b};\un{c})=\si_{\un{t}}(\un{a})$;
 \item[$\bullet$] for $\un{t}=(t)$ and $\un{r}=\un{s}=(r)$ we denote $\si_{t,r}(a,b,c)=\si_{\un{t};\un{r};\un{s}}(\un{a};\un{b};\un{c})$, where  $\un{a}=(a)$, $\un{b}=(b)$, $\un{c}=(c)$.
 \end{enumerate}
The element $\si_{t,r}(a,b,c)$ was introduced by Zubkov~\cite{ZubkovII}. Note that the definition from~\cite{ZubkovII} is different from our definition and their equivalence was established in Lemma~7.14 of~\cite{Lopatin_Orel}. More details can be found in Section~1.3 of~\cite{Lopatin_Orel}.    

As in Section~\ref{section2}, we define elements $F_t(\un{a})$, $P_{t,l}(b)$ of $\si\Y$, where $a_i,b\in\FF\Y$.

%------Rem5.2--------------------------------------------------
\begin{remark}\label{remark_O_notations}
Let $f\in\si\Y$. Taking the image of $f$ with respect to $\piLargeY$ ($\piSmallY{n}$, respectively), we can consider $f$ as an element of $\AlgLargeY$ ($\AlgSmallY{n}$, respectively). As an example, see formulations of Theorems~\ref{theo_O_main},~\ref{theo_Lopatin},~\ref{theo_O}. 
\end{remark}

%------Rem5.3--------------------------------------------------
\begin{remark}\label{remark_O_OK1}
Definition~4.1 together with Lemma 4.2 from~\cite{Lopatin_Orel} implies that $\si_{\un{t};\un{r};\un{s}}(\un{a};\un{b};\un{c})\in\AlgLargeY$ is a {\it partial linearization} of $\si_{t,r}(x,y,z)$, i.e., it is the coefficient of $\la_1^{t_1}\cdots \la_u^{t_u} \mu_1^{r_1}\cdots \mu_v^{r_v}\nu_1^{s_1}\cdots \nu_w^{s_w} $ in $\si_{t,r}(\la_1 a_1+\cdots+\la_u a_u, \mu_1 b_1+\cdots+\mu_v b_v, \nu_1 c_1+\cdots+\nu_w c_w)\in\AlgLargeY$ considered as a polynomial in $\la_1,\ldots,\la_u,\mu_1,\ldots,\mu_v,\nu_1,\ldots,\nu_w\in\FF$. 
\end{remark}

Note that $\si_{t,r}$ has certain symmetries. Namely,  
for $a,a_i,b,b_j,c,c_k\in\Y$ and $(a_1,\ldots,a_u)^T=(a_1^T,\ldots,a_u^T)$ we have 
\begin{eq}\label{eq_O_transpose}
\si_{\un{t};\un{r};\un{s}}(\un{a};\un{b};\un{c}) = \si_{\un{t};\un{r};\un{s}}(\un{a};\un{b}^T;\un{c}^T) = \si_{\un{t};\un{s};\un{r}}(\un{a}^T;\un{c};\un{b}) \;\text{ in }\; \AlgLargeY.
\end{eq}

Let $t,r\geq0$. To introduce $\chi_{t,r}$ and $\zeta_{t,r}$, analogues of the Cayley--Hamilton polynomial $\chi_t$ for the algebra $\CY_n$, we consider the quiver $\Q=\Q(x;y;z)$, where $x=x_1$, $y=x_2$, $z=x_3$, and 
\begin{enumerate}
\item[$\bullet$] denote by $L_{i,j}$ the set of pairwise different elements $e_1,\ldots,e_q\in \LA \Q\RA^{\#}$ with $e_1'=\cdots=e_q'=e_2''=\cdots=e_2''=1$ that satisfy $\mdeg(e_1)=\cdots=\mdeg(e_q)=(i,j,j)$; 
 
\item[$\bullet$] denote by $M_{i,j}$ the set of pairwise different paths $e_1,\ldots,e_q$ in $\Q$ with $e_1'=\cdots=e_q'=2$ and $e_1''=\cdots=e_q''=1$ that satisfy $\mdeg(e_1)=\cdots=\mdeg(e_q)=(i,j,j+1)$. 
\end{enumerate}
Then we consider the following elements of $\si\Y\otimes \FF\Y^{\#}$:
\begin{eq}\label{eq_O_chi}
\chi_{t,r}(x,y,z)=\sum_{i=0}^t \sum_{j=0}^r \si_{i,j}(x,y,z) \left( \sum_{e\in L_{t-i,r-j}} (-1)^{\xi}e \right),
\end{eq}
\begin{eq}\label{eq_O_zeta}
\zeta_{t,r}(x,y,z)=\sum_{i=0}^t \sum_{j=0}^r \si_{i,j}(x,y,z) \left( \sum_{e\in M_{t-i,r-j}} (-1)^{\xi}e \right),
\end{eq}%
where $\xi = i + \deg_{y}(e) + \deg_{z}(e)$. For short, here we have omitted $\otimes$. Note that $\chi_{0,0}(x,y,z)=1$, $\chi_{t,0}(x,y,z)=\chi_t(x)$, and $\zeta_{0,0}(x,y,z)=z^T-z$. For $a,b,c\in\FF\Y$ we define $\chi_{t,r}(a,b,c)$ and $\zeta_{t,r}(a,b,c)$ as the results of the corresponding substitutions. As in Remark~\ref{remark_O_notations}, we can consider $\chi_{t,r}(a,b,c)$ and $\zeta_{t,r}(a,b,c)$ as elements of $\AlgLargeY\otimes \FF\Y^{\#}$ as well as of $\AlgSmallY{n}\otimes \FF\Y^{\#}$. Connections between $\si_{t,r}$, $\chi_{t,r}$, $\zeta_{t,r}$ are given in Lemma~\ref{lemma_O_transpose} (see below).

%------Ex5.4--------------------------------------------------
\begin{example}\label{ex_54}
For $a,b,c\in\FF\Y$ and $\ov{b}=b-b^T$ the following equalities hold in $\si\Y$ and $\si\Y\otimes\FF\Y^{\#}$: 
\begin{enumerate}
 \item[$\bullet$] $\si_{0,1}(a,b,c) = -\tr(b\ov{c})$;
 
 \item[$\bullet$] $\si_{1,1}(a,b,c) = \tr(a\ov{b}\ov{c}) - \tr(a)\tr(b\ov{c})$;

 \item[$\bullet$] $\si_{0,2}(a,b,c) = \si_2(bc) + \si_2(bc^T) + \tr(bcbc^T) + \tr(bcb^Tc) - \tr(bcb^Tc^T) - \tr(bc)\tr(bc^T)$;

 \item[$\bullet$] $\chi_{0,1}(a,b,c) = \ov{b}\ov{c} - \tr(b\ov{c})$ and $\zeta_{1,0}(a,b,c) = - a^T\ov{c} - \ov{c}a + \tr(a)\ov{c}$; 
 
 \item[$\bullet$] $\chi_{1,1}(a,b,c) = a\ov{b}\ov{c} + \ov{b}a^T\ov{c} + \ov{b}\ov{c}a  - \tr(a)\ov{b}\ov{c} - \tr(b\ov{c})a - \tr(a\ov{b}\ov{c}) + \tr(a)\tr(b\ov{c})$;  
 
 \item[$\bullet$] $\zeta_{2,0}(a,b,c) =  - (a^T)^2\ov{c} - a^T\ov{c}a - \ov{c}a^2 + \tr(a)a^T\ov{c} + \tr(a)\ov{c}a - \si_2(a)\ov{c}$; 
 
  \item[$\bullet$] $\zeta_{0,1}(a,b,c) =  - \ov{c}\ov{b}\ov{c} + \tr(b\ov{c})\ov{c} $.
\end{enumerate}
\end{example}
\bigskip

The proof of the following theorem is given in Section~\ref{section7}.

%------Th5.5--------------------------------------------------
\begin{theo}\label{theo_O_main}
\begin{enumerate}
\item[1.] The ideal of relations $\TSmallY{n}$ for $\CY_n$ is generated by $\KSmallY{n}\otimes 1$ and 
\begin{enumerate}
\item[$\bullet$] $\chi_{t,r}(a,b,c)=0$ for $t+2r=n$; 
 
\item[$\bullet$] $\zeta_{t,r}(a,b,c)=0$ for $t+2r=n-1$; 
\end{enumerate}
where $a,b,c\in\FF\Y$.

\item[2.] The ideal of relations $\KSmallY{n}$ for $R^{O(n)}\simeq \AlgSmallY{n} / \KSmallY{n}$ is generated by 
\begin{enumerate}
\item[(a)] $\si_t(a+b)=F_t(a,b)$ for $1\leq t\leq n$, where $a,b\in\FF\Y$;  

\item[(b)] $\si_t(a^l)=P^{\pplus}_{t,l}(a)$ for $1\leq t\leq n$, $1<l\leq n$, where $a\in\Y$;

\item[(c)] $\si_t(ab)=\si_t(ba)$ for $1\leq t\leq n$, where $a,b\in\Y$;

\item[(d)] $\si_t(a)=\si_t(a^T)$ for $1\leq t\leq n$, where $a\in\Y$;

\item[(e)] $\si^{\pplus}_{\un{t};\un{r};\un{s}}(\un{a};\un{b};\un{c})=0$ for $n<|\un{t}|+2|\un{r}|\leq 2n$, where $\un{t}\in\NN_0^u$, $\un{r}\in\NN_0^v$, $\un{s}\in\NN_0^w$ ($u,v,w>0$) satisfy $|\un{r}|=|\un{s}|$  and $a_i,b_j,c_k\in\Y$ for all $i,j,k$.%
 \end{enumerate}
 \end{enumerate}
Moreover, we can assume that $\un{t}$, $\un{r}$, $\un{s}$ from relation~(e) satisfy condition~(\ref{eq_cond1}) and the vector $(\un{t},\un{r},\un{s})$ without zero entries satisfies conditions~(\ref{eq_cond2}) and~(\ref{eq_cond3}).

In particular, ideals $\TSmallY{n}$ and $\KSmallY{n}$ are finitely based. 
\end{theo}
\bigskip

Relations (a), (b), (c), (d) from Theorem~\ref{theo_O_main} are called {\it free} relations, because, being considered as elements of $\si\Y$, they belong to the kernel of $\PhiAbsY{n}$ for all $t\geq 1$, $l>1$ and do not depend on $n$.

%=====================================================================
%=====================================================================
%------Sec6--------------------------------------------------
\section{Large free algebra of $O(n)$-invariants}\label{section6}

We start this section with the known description of the ideal of relations $\KLargeY{n}$. We completed the proof of the following theorem in~\cite{Lopatin_free}, using results from~\cite{Lopatin_Orel} and the approach described in~\cite{ZubkovII}.

%------Th6.1--------------------------------------------------
\begin{theo}\label{theo_Lopatin} The ideal of relations $\KLargeY{n}$ for $R^{O(n)}\simeq \AlgLargeY / \KLargeY{n}$ is generated by $\si_{t,r}(a,b,c)=0$ for $t+2r>n$, $t,r\geq0$, and $a,b,c\in\FF\Y$.
\end{theo} 
\bigskip

The next lemma describes the large free algebra $\AlgLargeY$ as a quotient of the absolutely free algebra $\si\Y$. Its proof follows immediately from the proof of Lemma 3.1 from~\cite{Lopatin_free}.

%------L6.2--------------------------------------------------
\begin{lemma}\label{lemma_O_Lopatin} We have $\AlgLargeY\simeq \si\Y/ L'$ for the ideal $L'$ generated by
\begin{enumerate}
\item[(a)] $\si_t(a_1+\cdots+a_u)=F_{t}(a_1,\ldots,a_u)$, 

\item[(b)] $\si_t(a^l)=P_{t,l}(a)$,

\item[(c)] $\si_t(ab)=\si_t(ba)$, 

\item[(d)] $\si_t(a)=\si_t(a^T)$, 
\end{enumerate}
where $t>0$, $l,u>1$, $a_1,\ldots,a_u\in \FF\Y$, and $a,b\in \Y$. 
\end{lemma}
\bigskip

In this section we prove the following theorem together with Remark~\ref{remark_theo_O}:

%------Th6.3--------------------------------------------------
\begin{theo}\label{theo_O}
The ideal of relations $\KLargeY{n}$ for $R^{O(n)}\simeq \AlgLargeY / \KLargeY{n}$ is generated by 
\begin{enumerate}
\item[$\bullet$] $\si_{t,r}(a,b,c)=0$, where $n<t+2r\leq 2n$, $t,r\geq0$, and $a,b,c\in\FF\Y$; 

\item[$\bullet$] $\si_t(b)=0$, where $t>2n$ and $b\in\EY$. 
\end{enumerate}
\end{theo}

%------Rem6.4--------------------------------------------------
\begin{remark}\label{remark_theo_O}
We can reformulate Theorem~\ref{theo_O} as follows: the ideal $\KLargeY{n}$ is generated by 
\begin{enumerate}
\item[$\bullet$] $\si_{\un{t};\un{r};\un{s}}(\un{a};\un{b};\un{c})=0$ for $n<|\un{t}|+2|\un{r}|\leq 2n$, where $\un{t}\in\NN_0^u$, $\un{r}\in\NN_0^v$, $\un{s}\in\NN_0^w$ ($u,v,w>0$) satisfy condition~(\ref{eq_cond1}), $|\un{r}|=|\un{s}|$,  and $a_i,b_j,c_k\in\Y$ for all $i,j,k$; moreover, the vector $(\un{t},\un{r},\un{s})$ without zero entries satisfies conditions~(\ref{eq_cond2}) and~(\ref{eq_cond3});

\item[$\bullet$] $\si_t(b)=0$, where $t>n$ and $b\in\EY$. 
\end{enumerate}
\end{remark}
\bigskip

We split the proof of Theorem~\ref{theo_O} and Remark~\ref{remark_theo_O} into several lemmas. Given $l\geq0$, we denote by $J_{l}$ the ideal of $\AlgLargeY$ generated by $\si_{t,r}(a,b,c)$ satisfying $l=t+2r$ and  $a,b,c\in\FF\Y$. Since the field $\FF$ is infinite, Remark~\ref{remark_O_OK1} implies that elements $\si_{\un{t};\un{r};\un{s}}(\un{a};\un{b};\un{c})$ generate the ideal $J_l$ for $l=|\un{t}|+2|\un{r}|$, where $|\un{r}|=|\un{s}|$ and $a_i,b_j,c_k\in\Y$ for all $i,j,k$.  We write $J_{l}^{(p)}$ for the $\FF$-subspace of $\AlgLargeY$ spanned by $\si_{\un{t};\un{r};\un{s}}(\un{a};\un{b};\un{c})$ for $t_i,r_j,s_k\in\{1,p,p^2,\ldots\}$ and $a_i,b_j,c_k\in\Y$ satisfying $l=|\un{t}|+2|\un{r}|$ and $|\un{r}|=|\un{s}|$. Note that the ideal $J_l$ is closed with respect to partial linearizations.

%------Rem6.5--------------------------------------------------
\begin{remark}\label{remark_O_OK2}
Remark~\ref{remark_O_OK1}, the definition of $\AlgLargeY$ and Lemma~\ref{lemma_O_Lopatin} imply that for $\un{k}\in\NN^{l}$ ($l>0$), $k=|\un{k}|$, $\un{e}=(e_1,\ldots,e_l)$ with $e_i\in\Y$ for all $i$, and a letter $x$ we have
\begin{enumerate}
 \item[$\bullet$] $\si_{\un{t},\un{k};\un{r};\un{s}}(\un{a},\un{e};\un{b};\un{c})\in\AlgLargeY$ is a partial linearization of $\si_{\un{t},k;\un{r};\un{s}}(\un{a},x;\un{b};\un{c})$, where $|\un{r}|=|\un{s}|$; 
 
 \item[$\bullet$] $\si_{\un{t};\un{r},\un{k};\un{s}}(\un{a};\un{b},\un{e};\un{c})\in\AlgLargeY$ is a partial linearization of $\si_{\un{t};\un{r},k;\un{s}}(\un{a};\un{b},x;\un{c})$, where $|\un{s}|=|\un{r}|+k$;
  
  \item[$\bullet$] $\si_{\un{t};\un{r};\un{s},\un{k}}(\un{a};\un{b};\un{c},\un{e})\in\AlgLargeY$ is a partial linearization of $\si_{\un{t};\un{r};\un{s},k}(\un{a};\un{b};\un{c},x)$, where $|\un{r}|=|\un{s}|+k$.
\end{enumerate}
\end{remark}
\bigskip

We will use the following Lemma~\ref{lemma_O_45} together with Remark~\ref{remark_O_45}:

%------L6.6--------------------------------------------------
\begin{lemma}\label{lemma_O_45}
Given $\un{t}=(t_1,\ldots,t_u)$ and $\un{a}=(a_1,\ldots,a_u)$, we write $\un{t}'$ for $(1^{t_1},t_2,\ldots,t_u)$ and $\un{a}^{(i)}$ for $(\underbrace{a_1,\ldots,a_1}_{i},a_2,\ldots,a_u)$. Then in case $\FF=\QQ$ we have the following equalities
$$\si_{\un{t};\un{r};\un{s}}(\un{a};\un{b};\un{c})= 
\frac{1}{t_1!} \si_{\un{t}';\un{r};\un{s}}(\un{a}^{(t_1)};\un{b};\un{c}) = 
\frac{1}{r_1!} \si_{\un{t};\un{r}';\un{s}}(\un{a};\un{b}^{(r_1)};\un{c}) =
\frac{1}{s_1!} \si_{\un{t};\un{r};\un{s}'}(\un{a};\un{b};\un{c}^{(s_1)})$$
in $\AlgLargeY$.
\end{lemma}
\begin{proof}
Applying Remark~\ref{remark_O_OK2} instead of Remark~\ref{remark_GL_OK1}, we obtain the claim in the same way as we proved Lemma~\ref{lemma_GL_OK2}. 
\end{proof}

%------Rem6.7--------------------------------------------------
\begin{remark}\label{remark_O_45}
Let $\algA_{\FF}=\AlgLargeY$. We write $\algA_{\ZZ}$ for the set of all $f\in\algA_{\QQ}$ with integer coefficients. 
The natural surjective map $\ZZ\to \ZZ_p\subset \FF$ induces the well-defined homomorphism of rings $\algA_{\ZZ}\to \algA_{\FF}$. Here in case $p=0$ we assume that $\ZZ_p=\ZZ$.
\end{remark}
\bigskip

Similarly to a finite quiver $\Q(\un{x};\un{y};\un{z})$ from Section~\ref{section5} we can consider a quiver $\Q(a_1,a_2,\ldots;b_1,b_2,\ldots;c_1,c_2,\ldots)$ with infinitely many arrows, where $a_i,b_j,c_k\in\{x_1,x_2,\ldots\}$ are pairwise different letters for all $i,j,k>0$. Note that by path in a quiver we always mean a finite path.   

Consider a quiver $\Q_{I}$ and a quiver $\Q_{II}$ with infinitely many arrows:
$$\Q_{I}=\Q(x_0,x;y;z) \;\text{ and }\; \Q_{II}=\Q(x,e_i; y, u_i, v_i, w_{ij}; z)_{i,j>0}.$$

%------L6.8--------------------------------------------------
\begin{lemma}\label{lemma_O_sets1}
Define a homomorphism $\varphi:\LA \Q_{II}\RA \to \LA \Q_{I}\RA$ of semigroups as follows: 
$$\varphi(a^T)=\varphi(a)^T
\;\text{ and }\;
\varphi(a)=\left\{
\begin{array}{cl}
a,& \text{ if }a=x \text{ or } a=y \text{ or } a=z\\
x_0^ix ,& \text{ if } a=e_i\\
x_0^iy ,& \text{ if } a=u_i\\
x_0^iy^T ,& \text{ if } a=v_i^T\\
x_0^iy (x_0^T)^j ,& \text{ if } a=w_{ij}\\
\end{array}
\right.$$%
for an arrow $a$ of $\Q_{II}$ and $i,j>0$. We extend $\varphi$ to the map $\LA\Q_{II}\RA \sqcup\{x_0\} \to \LA\Q_{I}\RA$ by setting $\varphi(x_0)=x_0$. Then $\varphi$ induces the well-defined bijection $\ov{\varphi}:\LA \ov{\Q}_{II}\RA\sqcup\{\ov{x_0}\} \to \LA \ov{\Q}_{I}\RA$ of sets of $\sim$-equivalence classes of primitive elements.
\end{lemma}
\begin{proof}We split the proof into several statements. It is not difficult to see that $\varphi:\LA \Q_{II}\RA \to \LA \Q_{I}\RA$ is well-defined. We can extend $\varphi$ to the homomorphism $\path{\Q_{II}}\to\path{\Q_{I}}$, which we also denote by $\varphi$.

Denote by $\Omega_I$ the set of such elements $az,az^T\in\LA\Q_{I}\RA$ that a path $a$ satisfies $\deg_{z}(a)+\deg_{z^T}(a)=0$. In the same way we define $\Omega_{II}$. We claim that
\begin{eq}\label{eq_statement1}
 \varphi:\Omega_{II}\to\Omega_I \text{ is a bijection}.
\end{eq}%
An arbitrary element of $\Omega_I$ can be written as $ab^{\de}\!cz^{\eta}$, where $\de,\eta\in\{1,T\}$ and 
\begin{enumerate}
 \item[$\bullet$] $a=x^{j_0} x_0^{i_1} x^{j_1}\cdots x_0^{i_r} x^{j_r}$ for $j_0\geq0$, $i_1,j_1,\ldots,i_r,j_r>0$, $r\geq0$; in particular, $a=1$ in case $j_0=r=0$; 
 
 \item[$\bullet$] $b=x_0^i y (x_0^T)^j$ for $i,j\geq0$;
 
 \item[$\bullet$] $c=(x^T)^{k_1} (x_0^T)^{l_1}\cdots (x^T)^{k_s} (x_0^T)^{l_s}(x^T)^{k_{s+1}}$ for $k_1,l_1,\ldots,k_{s},l_{s}>0$, $k_{s+1}\geq0$, $s\geq0$; in particular, $c=1$ in case $s=k_1=0$.
\end{enumerate}%
Elements $a,b,c$ have unique preimages with respect to $\varphi$, namely, 
$$\varphi^{-1}(b)=\left\{
\begin{array}{cl}
y,& \text{ if }i=j=0\\
u_i ,& \text{ if } i>0,\, j=0\\
v_j ,& \text{ if } i=0,\, j>0\\
w_{ij} ,& \text{ if } i>0,\, j>0\\
\end{array}
\right.,$$%
$\varphi^{-1}(a) = x^{j_0} e_{i_1} x^{j_1-1}\cdots e_{i_r} x^{j_r-1}$, 
$\varphi^{-1}(c) = (x^T)^{k_1-1} e_{l_1}^T \cdots (x^T)^{k_s-1} e_{l_s}^T (x^T)^{k_{s+1}}$ for $a\neq1$ and $c\neq1$, respectively. Thus it is not difficult to see that $ab^{\de}\!cz^{\eta}$ has a unique preimage $\varphi^{-1}(ab^{\de}\!cz^{\eta})=\varphi^{-1}(a)(\varphi^{-1}(b))^{\de}\varphi^{-1}(c)z^{\eta}$. Statement~(\ref{eq_statement1}) is proven.

Consider paths $a,b\in\LA\Q_{II}\RA$. Then we claim that 
\begin{eq}\label{eq_statement2}
 a\sim b\text{ if and only if }\varphi(a)\sim\varphi(b).
\end{eq}%
Since $\varphi$ is a homomorphism, then $\varphi(a)\sim\varphi(b)$ follows from $a\sim b$.

Let $\varphi(a)\sim\varphi(b)$. Denote $\deg_z(a)+\deg_{z^T}(a)=\deg_z(b)+\deg_{z^T}(b)=r$. Note that if $r=0$, then Lemma~\ref{lemma_GL_sets} implies $a\sim b$. So we assume that $r>0$. Then $a\stackrel{c}{\sim} a_1\cdots a_r$ and $b\stackrel{c}{\sim} b_1\cdots b_r$ for $a_i,b_i\in\Omega_{II}$ ($1\leq i\leq r$).  

Assume $\varphi(a)\stackrel{c}{\sim}\varphi(b)$. Since $\varphi(a_i),\varphi(b_i)\in\Omega_I$ for all $i$, there is a cyclic permutation $\pi=(1,2,\ldots,r)^l\in S_r$ for some $l>0$ such that $\varphi(a_i)=\varphi(b_{\pi(i)})$. Statement~(\ref{eq_statement1}) implies $a_i=b_{\pi(i)}$. Therefore, $a\stackrel{c}{\sim}b$. If $\varphi(a)\stackrel{c}{\sim}\varphi(b)^T$, then $\varphi(a)\stackrel{c}{\sim}\varphi(b^T)$ and $a\sim b$ follows from the proven part of statement~(\ref{eq_statement2}).

Let $b\in\LA\Q_{I}\RA$ and $b\neq x_0^l$ for all $l>0$. Then we claim that 
\begin{eq}\label{eq_statement4}
\text{there exists an } a\in\LA\Q_{II}\RA \text{ satisfying }\varphi(a)\sim b.
\end{eq}%
Denote $\deg_z(b)+\deg_{z^T}(b)=r$. Note that if $r=0$, then Lemma~\ref{lemma_GL_sets} implies that $\varphi(a)=b$ for some $a\in\LA \Q_{II}\RA$. So we assume that $r>0$. Then $b\stackrel{c}{\sim} b_1\cdots b_r$ for $b_i\in\Omega_{I}$ ($1\leq i\leq r$). By statement~(\ref{eq_statement1}), there are $a_1,\ldots,a_r\in\LA \Q_{II}\RA$ such that $\varphi(a_i)=b_i$. Hence $\varphi(a_1\cdots a_r)\stackrel{c}{\sim} b$ and statement~(\ref{eq_statement4}) is proven. 

Consider $a\in\LA\Q_{II}\RA$. We claim that 
\begin{eq}\label{eq_statement5}
a \text{ is primitive if and only if } \varphi(a) \text{ is primitive}.
\end{eq}%
Let $\varphi(a)=b^l$ for $b\in\LA\Q_{I}\RA$ and $l>1$. If $\deg_z(a)+\deg_{z^T}(a)=0$, then the claim follows from Lemma~\ref{lemma_GL_sets}. Otherwise, $a\stackrel{c}{\sim} a_1\cdots a_r$ for $a_i\in\Omega_{II}$ and $r>0$. Hence, we obtain $b^l\stackrel{c}{\sim} \varphi(a_1)\cdots\varphi(a_r)$. By Remark~\ref{remark_primitive},  $\varphi(a_1)\cdots\varphi(a_r)=c^l$ for some $c\in\LA\Q_{I}\RA$. Since the last letter of $\varphi(a_r)$ is $z$ or $z^T$, the last letter of $c$ is also $z$ or $z^T$. Using the definition of $\Omega_{I}$ and the fact that $\varphi(a_i)\in\Omega_I$ for all $i$, we obtain that $c=\varphi(a_1)\cdots \varphi(a_s)$ for some $1\leq s<r$. Hence $\varphi(a_1)\cdots \varphi(a_r)=(\varphi(a_1)\cdots \varphi(a_s))^l$. In other words, $\varphi(a_i)=\varphi(a_{i+s})=\cdots=\varphi(a_{i+(l-1)s})$ for $1\leq i\leq s$. Statement~(\ref{eq_statement1}) implies $a_i=a_{i+s}=\cdots=a_{i+(l-1)s}$ for $1\leq i\leq s$, i.e., $a_1\cdots a_r=(a_1\cdots a_s)^l$. By Remark~\ref{remark_primitive}, $a$ is not primitive.  The converse claim is trivial.  

Now we can complete the proof of the theorem. Note that for $a\in\LA\Q_{II}\RA\sqcup\{x_0\}$ we have $\varphi(a)=x_0^l$ if and only if $l=1$ and $a=x_0$. By statements~(\ref{eq_statement2}) and (\ref{eq_statement5}), $\ov{\varphi}$ is a well-defined injective map. Statement~(\ref{eq_statement4}) implies that $\ov{\varphi}$ is a surjective map. 
\end{proof}

%------L6.9--------------------------------------------------
\begin{lemma}\label{lemma_O_key1}
Given $x_0,x,y,z\in\Y$ and $k,t,r\geq0$, we have $\si_{k,t;r;r}(x_0,x;y;z)\in J_{t+2r}$. 
\end{lemma}
\begin{proof}
We work in $\AlgLargeY$. Without loss of generality we can assume that $x_0,x,y,z\in\{x_1,x_2,\ldots\}$ are pairwise different letters. Assume that $e_i, y, u_i, v_i, w_{ij}, z$ ($i,j>0$) are pairwise different letters from $\{x_1,x_2,\ldots\}\backslash \{x_0,x,y,z\}$. In what follows, we use notations from Lemma~\ref{lemma_O_sets1}.  Let $\Upsilon_I$ be the set of finite multisubsets of $\LA\Q_I\RA$ and $\Upsilon_{II}$ be the set of finite multisubsets of $\LA\Q_{II}\RA \sqcup\{x_0\}$.  We define the $\sim$-equivalence on $\Upsilon_I$ naturally and denote by $\ov{\Upsilon}_I$ the set of all $\sim$-equivalence classes. Similarly we define $\ov{\Upsilon}_{II}$. Then Lemma~\ref{lemma_O_sets1} implies that $\ov{\varphi}:\ov{\Upsilon}_{II}\to \ov{\Upsilon}_I$ is a bijection. 

Let use recall that the definition of $\Omega(\un{t};\un{r};\un{s})$ was given in Section~\ref{section5}.  Assume that $\omega$ belongs to $\ov{\Upsilon}_I$ or $\ov{\Upsilon}_{II}$. Since we work in $\AlgLargeY$, the element $\si(\omega)$ is well-defined. For short, we write $\mdeg(\omega)$ for $\mdeg(\si(\omega))$. We refer to the entries of $\Delta=\mdeg(\omega)$ as follows:
\begin{enumerate}
 \item[$\bullet$] $\De=(\al_0,\al; \be;\ga)$ for $\omega\in\ov{\Upsilon}_I$, where $\De=\mdeg(x_0^{\al_0}x^{\al}  y^{\be}  z^{\ga})$;
 
 \item[$\bullet$] $\De=(\al_0,\al,\al_i; \be,\la_i,\mu_i,\nu_{ij};\ga)_{i,j>0}$ for $\omega\in\ov{\Upsilon}_{II}$, where $\De$ is equal to 
$$\mdeg(x_0^{\al_0}x^{\al} y^{\be} z^{\ga} \prod_i x_i^{\al_i} u_i^{\la_i} v_i^{\mu_i} \prod_{j} w_{ij}^{\nu_{ij}} ).$$ 
Here we assume that only finitely many elements from $\{\al_i,\la_i,\mu_i,\nu_{ij}\}_{i,j>0}$ are non-zero.  
\end{enumerate}%  
In the first case (the second case, respectively) we say that $\De$ is a multidegree of type I (type II, respectively). By the definition, 
\begin{eq}\label{eq1_lemma_O_key1}
 \si_{k,t;r;r}(x_0,x;y;z)=(-1)^{k+t} \sum_{\omega\in \ov{\Omega}_I} \si(\omega),
\end{eq}%
where $\ov{\Omega}_I=\ov{\Omega}(k,t;r;r)=\{\omega\in\ov{\Upsilon}_{I}\,|\,\mdeg(\omega)=(k,t;r;r)\}$. For $\ov{\Omega}_{II}=\{\omega\in\ov{\Upsilon}_{II}\,|\,\mdeg(\varphi(\omega))=(k,t;r;r)\}$ an isomorphism of sets $\ov{\Omega}_{II}\simeq \ov{\Omega}_I$ is determined by the restriction of $\ov{\varphi}$. 

Given a multidegree $\De$ of type II, we denote $\ov{\Omega}_{II}^{\De}=\{\omega\in\ov{\Upsilon}_{II}\,|\,\mdeg(\omega)=\De\}$ and $\varphi(\De)=\mdeg(\varphi(x_0^{\al_0}x^{\al} y^{\be} z^{\ga} \prod_i x_i^{\al_i} u_i^{\la_i} v_i^{\mu_i} \prod_{j} w_{ij}^{\nu_{ij}}))$. Thus 
\begin{eq}\label{eq2_lemma_O_key1}
 \ov{\Omega}_{II}=\bigsqcup \ov{\Omega}_{II}^{\Delta}, 
\end{eq}%
where the union ranges over $\Delta$ of type II satisfying $\varphi(\Delta)=(k,t;r;r)$. Consequently applying formula~(\ref{eq1_lemma_O_key1}), the isomorphism $\ov{\Omega}_{II}\simeq \ov{\Omega}_I$, and formula~(\ref{eq2_lemma_O_key1}) we obtain
$$\si_{k,t;r;r}(x_0,x;y;z)=(-1)^{k+t}\!\!\!\! \sum_{\varphi(\Delta)=(k,t;r;r)}\; \sum_{\omega\in\ov{\Omega}_{II}^{\Delta} } \si(\varphi(\omega)).$$%
Since 
$$\deg_y(\varphi(c))+\deg_z(\varphi(c)) = \deg_y(c) + \sum_i\deg_{u_i}(c) + \sum_i\deg_{v_i}(c) + \sum_{ij}\deg_{w_{ij}}(c) + \deg_z(c)$$% 
for all $c\in\LA \Q_{II}\RA$, we have $\si(\varphi(\omega))=\varphi(\si(\omega))$ for $\omega\in\ov{\Omega}_{II}^{\Delta}$. Therefore, 
$$\sum_{\omega\in\ov{\Omega}_{II}^{\Delta} } \si(\varphi(\omega))=(-1)^{\al_0+|\Delta'|} \si_{\al_0}(x_0)\,  \varphi( \si_{\Delta'}(x,e_i; y, u_i, v_i, w_{ij}; z)_{i,j>0}),$$%
where $\Delta'$ stands for $(\al,\al_i; \be,\la_i,\mu_i,\nu_{ij};\ga)_{i,j>0}$. 
The condition $\varphi(\Delta)=(k,t;r;r)$ implies $|\Delta'|=t+2r$. Thus, 
\begin{eq}\label{eq3_lemma_O_key1}
\begin{array}{c}
\si_{k,t;r;r}(x_0,x;y;z) \\
=\sum\limits(-1)^{\al_0+k}\si_{\al_0}(x_0)\, \si_{\Delta'}(x,x_0^i x; y, x_0^i y, y (x_0^T)^i, x_0^i y (x_0^T)^j; z)_{i,j>0},\\
\end{array}
\end{eq}%
where the sum ranges over $\Delta$ of type II satisfying $\varphi(\Delta)=(k,t;r;r)$. 
The required is proven.
\end{proof}

%------Ex6.10--------------------------------------------------
\begin{example}\label{ex_O1}
For $x_0,x,y,z\in\Y$ and $t,r>0$ the following equalities of $\AlgLargeY$ are partial cases of key formula~(\ref{eq3_lemma_O_key1}) from the proof of Lemma~\ref{lemma_O_key1}: 
\begin{enumerate}
\item[$\bullet$] $\si_{1,t;r;r}(x_0,x;y;z) = 
\tr(x_0)\si_{t;r;r}(x;y;z) 
- \si_{t-1,1;r;r}(x,x_0x;y;z)$ 

$- \si_{t;r-1,1;r}(x;y,x_0y;z) 
- \si_{t;r-1,1;r}(x;y,yx_0^T;z) \in J_{t+2r}$; 
 
\item[$\bullet$] $\si_{(2,0;2;2)}(x_0,x;y;z) = 
\si_2(x_0)\si_{0;2;2}(x;y;z)$

$- \tr(x_0)\si_{0;1,1;2}(x;y,x_0y;z) - \tr(x_0)\si_{0;1,1;2}(x;y,yx_0^T;z)$

$+ \si_{0;2;2}(x;x_0y;z) + \si_{0;2;2}(x;yx_0^T;z)$

$+ \si_{0;1,1;2}(x;y,x_0^2y;z) + \si_{0;1,1;2}(x;y,y(x_0^T)^2;z)$

$+ \si_{0;1,1;2}(x;x_0y,yx_0^T;z)
+ \si_{0;1,1;2}(x;y,x_0yx_0^T;z) \in J_4$.
\end{enumerate}
\end{example}
\bigskip

Consider quivers $\Q_{III}=\Q(x;y_0,y;z)$ and $\Q_{IV}=\Q(x,e_1,e_2; y, y_1; z)$.

%------L6.11--------------------------------------------------
\begin{lemma}\label{lemma_O_sets2}
Define a homomorphism $\varphi:\LA \Q_{IV}\RA \to \LA \Q_{III}\RA$ of semigroups as follows: 
$$\varphi(a^T)=\varphi(a)^T
\;\text{ and }\;
\varphi(a)=\left\{
\begin{array}{cl}
a,& \text{ if }a=x \text{ or } a=y \text{ or } a=z\\
y_0z ,& \text{ if } a=e_1\\
y_0z^T ,& \text{ if } a=e_2\\
y_0x^T ,& \text{ if } a=y_1\\
\end{array}
\right.$$%
for an arrow $a$ of $\Q_{IV}$.  Then $\varphi$ induces the well-defined bijection $\ov{\varphi}:\LA \ov{\Q}_{IV}\RA \to \LA \ov{\Q}_{III}\RA$ of sets of $\sim$-equivalence classes of primitive elements.
\end{lemma}
\begin{proof}It is not difficult to see that $\varphi:\LA \Q_{IV}\RA \to \LA \Q_{III}\RA$ is well-defined. Denote by $\Omega_{III}$ the set of elements $a\in\LA\Q_{III}\RA$ satisfying $\deg_z(a)+\deg_{z^T}(a)=1$ and denote by $\Omega_{IV}$ the set of elements $a\in\LA\Q_{IV}\RA$ satisfying 
$$\deg_z(a)+\deg_{z^T}(a)+\sum_{i=1}^2(\deg_{e_i}(a)+\deg_{e_i^T}(a))=1.$$% 
We claim that 
\begin{eq}\label{eq_statement6}
 \ov{\varphi}: \ov{\Omega}_{IV}\to\ov{\Omega}_{III} \text{ is a bijection of $\sim$-equivalence classes}.
\end{eq}%
An arbitrary element of $\Omega_{III}$ is $\sim$-equivalent to $a=x^i b (x^T)^j z^{\de}$ for some $i,j\geq0$, $\de\in\{1,T\}$, and $b\in\{y_0,y\}$. If $\deg_{y_0}(a)=0$, then $a$ has a unique preimage $\varphi^{-1}(a)=a$. Otherwise, $b=y_0$ and $a$ also has a unique preimage, namely, 
$$\varphi^{-1}(a)=\left\{
\begin{array}{cl}
x^i e_1 ,& \text{ if } j=0 \text{ and }\de=1\\
x^i e_2 ,& \text{ if } j=0 \text{ and }\de=T\\
x^i y_1 (x^T)^{j-1} z^{\de} ,& \text{ if } j>0\\
\end{array}
\right..$$%
Statement~(\ref{eq_statement6}) is proven. 

Applying statement~(\ref{eq_statement6}) instead of statement~(\ref{eq_statement1}), we complete the proof of the lemma in a similar way as the proof of Lemma~\ref{lemma_O_sets1}.
\end{proof}

%------L6.12--------------------------------------------------
\begin{lemma}\label{lemma_O_key2}
Given $x,y_0,y,z\in\Y$ and $t,r,s\geq0$, we have $\si_{t;r,s;r+s}(x;y_0,y;z)\in J_{t+r+2s}$. 
\end{lemma}
\begin{proof}
We work in $\AlgLargeY$. Without loss of generality we can assume that $x,y_0,y,z\in\{x_1,x_2,\ldots\}$ are pairwise different letters. Assume that $e_1,e_2, y_1$ are pairwise different letters from $\{x_1,x_2,\ldots\}\backslash \{x,y_0,y,z\}$. In what follows, we use notations from Lemma~\ref{lemma_O_sets2}.  Let $\Upsilon_{III}$ be the set of finite multisubsets of $\LA\Q_{III}\RA$ and $\Upsilon_{IV}$ be the set of finite multisubsets of $\LA\Q_{IV}\RA$.  Then Lemma~\ref{lemma_O_sets2} implies that $\ov{\varphi}:\ov{\Upsilon}_{IV}\to \ov{\Upsilon}_{III}$ is a bijection of sets of multisets of $\sim$-equivalence classes. 

Assume that $\omega$ belongs to $\ov{\Upsilon}_{III}$ or $\ov{\Upsilon}_{IV}$. Since we work in $\AlgLargeY$, the element $\si(\omega)$ is well-defined. For short, we write $\mdeg(\omega)$ for $\mdeg(\si(\omega))$. We refer to the entries of $\Delta=\mdeg(\omega)$ as follows:
\begin{enumerate}
 \item[$\bullet$] $\De=(\al; \be_0, \be;\ga)$ for $\omega\in\ov{\Upsilon}_{III}$, where $\De=\mdeg(x^{\al}  y_0^{\be_0} y^{\be} z^{\ga})$;
 
 \item[$\bullet$] $\De=(\al,\al_1,\al_2; \be,\be_1;\ga)$ for $\omega\in\ov{\Upsilon}_{IV}$, where $\De=\mdeg(x^{\al} x_1^{\al_1} x_2^{\al_2} y^{\be} y_1^{\be_1} z^{\ga})$.
\end{enumerate}%
In the first case (the second case, respectively) we say that $\De$ is a multidegree of type III (type IV, respectively). By the definition, 
\begin{eq}\label{eq1_lemma_O_key2}
 \si_{t;r,s;r+s}(x;y_0,y;z)=(-1)^{t}\!\! \sum_{\omega\in \ov{\Omega}_{III}} \si(\omega),
\end{eq}%
where $\ov{\Omega}_{III}=\ov{\Omega}(t;r,s;r+s)=\{\omega\in\ov{\Upsilon}_{III}\,|\,\mdeg(\omega) = (t;r,s;r+s)\}$. For $\ov{\Omega}_{IV} = \{\omega\in\ov{\Upsilon}_{IV}\,|\,\mdeg(\varphi(\omega))=(t;r,s;r+s)\}$ an isomorphism of sets $\ov{\Omega}_{IV}\simeq \ov{\Omega}_{III}$ is determined by the restriction of $\ov{\varphi}$. 

Given a multidegree $\De$ of type IV, we denote $\ov{\Omega}_{IV}^{\De}=\{\omega\in\ov{\Upsilon}_{IV}\,|\,\mdeg(\omega)=\De\}$ and $\varphi(\De)=\mdeg(\varphi(x^{\al} x_1^{\al_1} x_2^{\al_2} y^{\be} y_1^{\be_1} z^{\ga}))$. Thus 
\begin{eq}\label{eq2_lemma_O_key2}
 \ov{\Omega}_{IV}=\bigsqcup \ov{\Omega}_{IV}^{\Delta}, 
\end{eq}%
where the union ranges over $\Delta$ of type IV satisfying $\varphi(\Delta)=(t;r,s;r+s)$. Consequently applying formula~(\ref{eq1_lemma_O_key2}), the isomorphism $\ov{\Omega}_{IV}\simeq \ov{\Omega}_{III}$, and formula~(\ref{eq2_lemma_O_key2}) we obtain
$$\si_{t;r,s;r+s}(x;y_0,y;z)=(-1)^{t}\!\!\!\! \sum_{\varphi(\Delta)=(t;r,s;r+s)}\; \sum_{\omega\in\ov{\Omega}_{IV}^{\Delta} } \si(\varphi(\omega)).$$%
Since $\deg_{y_0}(\varphi(c))+\deg_y(\varphi(c))+\deg_z(\varphi(c))$ is equal to 
$$\left(\deg_y(c) + \deg_{y_1}(c) + \deg_z(c)\right) + 
 \left( 2\deg_{e_1}(c) + \deg_{e_2}(c) + \deg_{e_2^T}(c)\right)
$$% 
for all $c\in\LA\Q_{IV}\RA$, we have $\si(\varphi(\omega))=(-1)^{\al_2}\varphi(\si(\omega))$ for $\omega\in\ov{\Omega}_{IV}^{\Delta}$. Therefore,
$$\sum_{\omega\in\ov{\Omega}_{IV}^{\Delta} } \si(\varphi(\omega))=(-1)^{|\Delta|+\al_2}   \varphi(\si_{\Delta}(x,e_1, e_2; y, y_1; z)).$$%
A multidegree $\Delta$ of type IV satisfies the equality $\varphi(\Delta)=(t;r,s;r+s)$ if and only if 
$$\begin{array}{rcl}
 \al+\be_1 &=& t\\
 \al_1+\al_2+\be_1 &=& r\\
 \be&=&s\\
 \ga+\al_1+\al_2&=&r+s\\
\end{array};
$$%
in particular, $|\Delta|=t+r+2s$. Thus, 
\begin{eq}\label{eq3_lemma_O_key2}
\si_{t;r,s;r+s}(x;y_0,y;z) = \sum\limits (-1)^{\al_2+r} \si_{\Delta}(x, y_0 z, y_0 z^T; y, y_0x^T;  z),
\end{eq}%
where the sum ranges over $\Delta$ of type IV satisfying $\varphi(\Delta)=(t;r,s;r+s)$. The required is proven.
\end{proof}

%------Ex6.13--------------------------------------------------
\begin{example}\label{ex_O2}
For $x,y_0,y,z\in\Y$ and $t>0$, $s\geq0$ the following equality of $\AlgLargeY$ is a partial case of key formula~(\ref{eq3_lemma_O_key2}) from the proof of Lemma~\ref{lemma_O_key2}: 
$$\begin{array}{c}
\si_{t;1,s;s+1}(x;y_0,y;z) = -\si_{t,1;s;s}(x,y_0z;y;z) + \si_{t,1;s;s}(x,y_0z^T;y;z)\\
- \si_{t-1;s,1;s+1}(x;y,y_0x^T;z)  \in J_{t+2s+1}.\\
\end{array}
$$ 
\end{example}

%------L6.14--------------------------------------------------
\begin{lemma}\label{lemma_O_key3}
Given $\un{t}\in\NN^u_0$, $\un{r}\in\NN^v_0$, $\un{s}\in\NN^w_0$ with $t=|\un{t}|$ and $r=|\un{r}|=|\un{s}|$ ($u,v,w>0$). Then for $f=\si_{\un{t};\un{r};\un{s}}(\un{a};\un{b};\un{c})\in \AlgLargeY$, where $a_i,b_j,c_k\in\Y$, we have
 \begin{enumerate}
  \item[$\bullet$] $f\in J_{t+2r-t_1}$;
  
  \item[$\bullet$] $f\in J_{t+2r-r_1}$ and $f\in J_{t+2r-s_1}$.
 \end{enumerate}
\end{lemma}
\begin{proof}
Assume $t_1>0$. By Remark~\ref{remark_O_OK2}, $f$ is equal to the coefficient of $\la_2^{t_2}\cdots\la_u^{t_u} \mu_1^{r_1}\cdots\mu_v^{r_v} \nu_1^{s_1}\cdots\nu_w^{s_w}$ in $h=\si_{\un{k};r;r}(a_1,\la_2 a_2+\cdots+\la_u a_u; \mu_1 b_1+\cdots+\mu_v b_v; \nu_1 c_1+\cdots+\nu_w c_w)$, where $\un{k}=(t_1,t-t_1)$ and $\la_2,\ldots,\nu_w\in\FF$. Lemma~\ref{lemma_O_key1} implies that $h$ belongs to $J_{t+2r-t_1}$. Since the ideal $J_{t+2r-t_1}$ is closed with respect to partial linearizations, we obtain $f\in J_{t+2r-t_1}$.  

Let $r_1>0$. Applying Lemma~\ref{lemma_O_key2} instead of Lemma~\ref{lemma_O_key1}, we obtain $f\in J_{t+2r-r_1}$ by the same reasoning as above. Formula~(\ref{eq_O_transpose}) concludes the proof.
\end{proof}

%------L6.15--------------------------------------------------
\begin{lemma}\label{lemma_O_key4}
Given $\un{t}\in\NN_0^u$, $\un{r}\in\NN_0^v$, $\un{s}\in\NN_0^w$ with $t=|\un{t}|$ and $r=|\un{r}|=|\un{s}|$. Then $\si_{\un{t};\un{r};\un{s}}(\un{a};\un{b};\un{c})$ belongs to $J_{t+2r}^{(p)}$ for $a_i,b_j,c_k\in\Y$.  
\end{lemma}
\begin{proof} We repeat the proof of Lemma~\ref{lemma_GL_key2}, applying Lemma~\ref{lemma_O_45} together with Remark~\ref{remark_O_45} instead of Lemma~\ref{lemma_GL_OK2}. 
\end{proof} 

Now we can prove Theorem~\ref{theo_O} and Remark~\ref{remark_theo_O}:
\begin{proof}
We work in $\AlgLargeY$. Since the field $\FF$ is infinite, Theorem~\ref{theo_Lopatin} together with Remark~\ref{remark_O_OK1} implies that $\KLargeY{n}$ is generated by 
\begin{enumerate}
\item[(a)] $\si_{\un{t},\un{r},\un{s}}(\un{a},\un{b},\un{c})=0$ for $|\un{t}|+2|\un{r}|>n$, where $\un{t}\in\NN_0^u$, $\un{r}\in\NN_0^v$, $\un{s}\in\NN_0^w$ ($u,v,w>0$) satisfy $|\un{r}|=|\un{s}|$ and $a_i,b_j,c_k\in\Y$;

\item[(b)] $\si_t(b)=0$ for $t>n$, where $b\in\Y$.
\end{enumerate}

Applying Lemmas~\ref{lemma_O_key3} and~\ref{lemma_O_key4} instead of Lemmas~\ref{lemma_GL_key} and~\ref{lemma_GL_key2}, respectively, we complete the proof in the same way as we proved Theorem~\ref{theo_GL} at the end of Section~\ref{section3}.
\end{proof}

%------L6.16--------------------------------------------------
\begin{lemma}\label{lemma_O_fb}
If $p=0$, then the ideal $\KLargeY{n}\vartriangleleft \AlgLargeY$ is generated by $\si_{t,r}(a,b,c)=0$ for $t+2r=n+1$, $t,r\geq0$, and $a,b,c\in\FF\Y$; in particular, $\KLarge{n}$ is finitely based.

If $p>0$, then the ideal $\KLargeY{n}\vartriangleleft \AlgLargeY$ is not finitely based. 
\end{lemma}
\begin{proof}
Using Theorem~\ref{theo_O}, Remarks~\ref{remark_theo_O},~\ref{remark_O_OK1} and Lemma~\ref{lemma_O_Lopatin}, instead of Theorem~\ref{theo_GL}, Remarks~\ref{remark_theo_GL},~\ref{remark_GL_OK1} and Lemma~\ref{lemma_GL_Donkin}, respectively, we prove this lemma in the same way as we proved Lemma~\ref{lemma_GL_fb}.
\end{proof}

%=====================================================================
%=====================================================================
%------Sec7--------------------------------------------------
\section{Proof of Theorem~\ref{theo_O_main}}\label{section7}

Assume that 
\begin{eq}\label{eq5} 
f=\sum_i \al_i f_i a_i\in \AlgLargeY\otimes \FF\Y^{\#}
\end{eq}%
for $\al_i\in\FF$, $f_i=\si_{t_{i1}}(b_{i1})\cdots \si_{t_{ir_i}}(b_{ir_i})\in \AlgLargeY$,  $b_{ij}\in\Y$, $a_i\in\Y^{\#}$. 

If $a_i\in\Y$ for all $i$, then we write $\tr(f)$ for $\sum_i \al_i f_i \tr(a_i)\in \AlgLargeY\otimes \FF\Y^{\#}$. Note that we do not define $\tr(f)$ for $f\in\AlgLargeY\otimes 1\subset \AlgLargeY\otimes \FF\Y^{\#}$. We say that $f$ does not contain a letter $x$ if $\deg_x(f_i)=0$ and $\deg_x(a_i)=0$ for all $i$.

%------L7.1--------------------------------------------------
\begin{lemma}\label{lemma_71}
Let $f\in \AlgLargeY\otimes \FF\Y^{\#}$ do not contain letters $x$ and $x^T$. Then 
\begin{enumerate}
\item[1)] $f\in \TLargeY{n}$ if and only if $\tr(fx)\in \KLargeY{n}$;
 
\item[2)] $f=0$ if and only if $\tr(fx)=0$.
\end{enumerate}
\end{lemma}
\begin{proof}
1) Note that $\tr(fx)\in\KLargeY{n}$ if and only if $\tr(\PsiLargeY{n}(f)X)=0$ for the generic $n\times n$ matrix $X$ corresponding to $x$. Since the trace bilinear form $\tr:M_n(\FF)\times M_n(\FF)\to \FF$ is nondegenerate, the last condition is equivalent to the fact that $\PsiLargeY{n}(f)=0$. The required is proven.
\smallskip

2) Let $f$ be given by formula~(\ref{eq5}). Since the equality $\al f_i\tr(a_i x) = f_j\tr(a_j x)$ in $\AlgLargeY\otimes \FF\Y^{\#}$ for some $\al\in\FF$ implies $\al=1$, $f_i=f_j$, and $a_i=a_j$, we obtain that the equality $f=0$ follows from 
$\tr(fx)=0$. 
\end{proof}

Analogues of formula~(\ref{eq_O_transpose}) hold for $\chi_{t,r}$ and $\zeta_{t,r}$.

%------L7.2--------------------------------------------------
\begin{lemma}\label{lemma_O_transpose}
For $x,a,b,c\in\Y$ we have 
\begin{enumerate}
\item[1)] $\si_{t,1;r;r}(a,x;b;c) = (-1)^{t}\tr(\chi_{t,r}(a,b,c)x)\;$ and $\;\si_{t;r,1;r+1}(a;b,x;c) = (-1)^{t} \tr(\zeta_{t,r}(a,b,c)x)$ in $\AlgLargeY$;

\item[2)] $\chi_{t,r}(a,b,c)^T=\chi_{t,r}(a^T,c,b)=\chi_{t,r}(a,b^T,c^T)^T$ in $\AlgLargeY\otimes \FF\Y^{\#}$;

\item[3)] $\zeta_{t,r}(a,b,c)^T=\zeta_{t,r}(a,b^T,c^T)$ in $\AlgLargeY\otimes \FF\Y^{\#}$.
\end{enumerate}
\end{lemma}
\begin{proof}
Part~1 follows from the definitions and parts~2,~3 are consequences of formula~(\ref{eq_O_transpose}), part~1 of the lemma, and part~2 of Lemma~\ref{lemma_71}.
\end{proof}

%------L7.3--------------------------------------------------
\begin{lemma}\label{lemma_73}
The ideal of relations $\TLargeY{n}$ for $\CY_n$ is generated by $\KLargeY{n}\otimes 1$ and 
\begin{enumerate}
\item[$\bullet$] $\chi_{t,r}(a,b,c)=0$ for $t+2r=n$; 
 
\item[$\bullet$] $\zeta_{t,r}(a,b,c)=0$ for $t+2r=n-1$; 
\end{enumerate}
where $a,b,c\in\FF\Y$.
\end{lemma}
\begin{proof}
By part~1 of Lemma~\ref{lemma_71}, an element $f\in\AlgLargeY\otimes \FF\Y^{\#}$ belongs to $\TLargeY{n}$ if and only if $\tr(fx)\in\KLargeY{n}$ for such a letter $x$ that neither $x$ nor $x^T$ is not contained in $f$. Since $\deg_{x}\tr(fx)+\deg_{x^T}\tr(fx)=1$, Theorem~\ref{theo_O} together with Remarks~\ref{remark_theo_O},~\ref{remark_O_OK2} implies that the last condition holds if and only if $\tr(fx)$ belongs to the ideal of $\AlgLargeY$, generated by 
\begin{enumerate}
 \item[(a)] $\si_{t,1;r;r}(a,e;b;c)=0$ for $t+2r=n$;
 
 \item[(b)] $\si_{t;r,1;r+1}(a;b,e;c)=0$ and $\si_{t;r+1;r,1}(a;b;c,e)=0$ for $t+2r=n-1$;
 
 \item[(c)] $h\tr(e)$ for $h\in\KLargeY{n}$; 
\end{enumerate}
where $a,b,c\in\FF\Y$, $e=e_1 x^{\de} e_2$ for $e_1,e_2\in\Y^{\#}$ and $\de\in\{1,T\}$. By formula~(\ref{eq_O_transpose}), $\si_{t;r+1;r,1}(a;b;c,e)=\si_{t;r,1;r+1}(a^T;c^T,e^T;b^T)$ in $\AlgLargeY$. Thus, parts~1,~2,~3 of Lemma~\ref{lemma_O_transpose} imply that elements~(a) and~(b) of $\AlgLargeY$ coincide with elements
\begin{enumerate}
 \item[$\bullet$] $\pm\tr(\chi_{t,r}(a,b,c)e_1xe_2)=0$ for $t+2r=n$;
 
 \item[$\bullet$] $\pm\tr(\zeta_{t,r}(a,b,c)e_1xe_2)=0$ for $t+2r=n-1$; 
\end{enumerate}
where $a,b,c\in\FF\Y$ and $e_1,e_2\in\Y^{\#}$. Finally, part~2 of Lemma~\ref{lemma_71} completes the proof.
\end{proof} 

Now we can complete the proof of Theorem~\ref{theo_O_main}.
\begin{proof}
Applying Lemma~\ref{lemma_73}, we prove part~1 of Theorem~\ref{theo_O_main} exactly in the same way as we proved part~1 of Theorem~\ref{theo_O_main} at the end of Section~\ref{section4}. Using Lemma~\ref{lemma_O_Lopatin}, Theorem~\ref{theo_O}, Remark~\ref{remark_theo_O} instead of Lemma~\ref{lemma_GL_Donkin}, Theorem~\ref{theo_GL}, Remark~\ref{remark_theo_GL}, respectively, we repeat the reasoning from Section~\ref{section4} to prove part~2 of Theorem~\ref{theo_O_main}.
\end{proof}

%=====================================================================
%=====================================================================

\section*{Acknowledgements}
This paper was written during author's visit to Bielefeld University, sponsored by CRC 701 ``Spectral Structures and Topological Methods in Mathematics''. The author is grateful for this support. The author is also grateful to Professor Claus Michael Ringel for hospitality. This paper has also been partially supported by grants of Ministry of Education and Science of Russia \textnumero 14.B37.21.0359 and \textnumero 0859.


\begin{thebibliography}{99}
\bibitem{Amitsur_1980}S.A. Amitsur, {\it On the characteristic polynomial of a sum of
matrices}, Linear Mult.~Algebra {\bf 8} (1980), 177--182.

\bibitem{Donkin_1992a}S. Donkin, {\it Invariants of several matrices}, Invent.~Math. {\bf 110} (1992), 389--401.

\bibitem{Donkin_1993a}S. Donkin, {\it Invariant functions on matrices}, Math.~Proc.~Cambridge Philos.~Soc. {\bf 113} (1993), 23--43.

\bibitem{DKZ_2002}M. Domokos, S.G. Kuzmin, A.N. Zubkov, {\it Rings of
matrix invariants in positive characteristic}, J.~Pure
Appl.~Algebra {\bf 176} (2002), 61--80.  

\bibitem{Koshlukov_2001} P. Koshlukov, {\it Basis of the identities of the matrix algebra of order two over a field of characteristic $p\neq 2$}, J.~Algebra {\bf 241} (2001), 410--434.

\bibitem{Koshlukov_2004} J. Colombo, P. Koshlukov,  {\it Central polynomials in the matrix algebra of order two},  Linear Algebra Appl. {\bf 377} (2004), 53--67.

\bibitem{Koshlukov_2005} J. Colombo, P. Koshlukov,  {\it Identities with involution for the matrix algebra of order two in characteristic $p$}, Israel J.~Math. {\bf 146} (2005), 337--355. 

\bibitem{Lopatin_bplp}A.A. Lopatin, {\it On block partial linearizations of the pfaffian},
Linear Algebra Appl. {\bf 426/1} (2007), 109--129.

\bibitem{Lopatin_free}A.A. Lopatin, {\it Free relations for matrix invariants in modular case}, J.~Pure Appl.~Algebra {\bf 216} (2012), 427--437.

\bibitem{Lopatin_Orel}A.A. Lopatin, {\it Relations between $O(n)$-invariants of several matrices}, to appear in Algebra Repr.~Theory, arXiv: 0902.4266.

\bibitem{Procesi_1976}C. Procesi, {\it The invariant theory of $n\times n$ matrices}, Adv.~Math. {\bf 19} (1976), 306--381.

\bibitem{Samoilov_2007}L.M. Samoilov, {\it On the nilindex of the radical of a relatively free associative algebra}, Mat.~Zametki {\bf 82} (2007), No.~4, 583--592 (Russian); Math.~Notes {\bf 82} (2007), No.~4, 522--530 (Engl.~transl.). 

\bibitem{Samoilov_2008}L.M. Samoilov, {\it On the radical of a relatively free associative algebra over fields of positive characteristic},  Mat.~Sb.~{\bf 199} (2008), No.~5, 81--126 (Russian); Sb.~Math.~{\bf 199} (2008), No.~5, 707--753 (Engl.~transl.)

\bibitem{Razmyslov_1973}Yu.P. Razmyslov, {\it The existence of a finite basis for the identities of the matrix algebra of order two over a field of characteristic zero}, Algebra i Logika {\bf 12} (1973), No. 1, 83--113 (Russian); Algebra and Logic {\bf 12} (1973), 47--63 (Engl. transl.).

\bibitem{Razmyslov_1974}Yu.P. Razmyslov, {\it Trace identities of full matrix
algebras over a field of characteristic zero}, Izv.~Akad.~Nauk SSSR Ser.~Mat. {\bf 38}
(1974), No.~4, 723--756 (Russian).

\bibitem{Sibirskii_1968}K.S. Sibirskii, {\it Algebraic invariants of a system of
matrices}, Sibirsk.~Mat.~Zh. {\bf 9} (1968), No.~1, 152--164 (Russian).

\bibitem{Zubkov_1996}A.N. Zubkov, {\it On a generalization of the Razmyslov--Procesi
theorem}, Algebra and Logic {\bf 35} (1996), No.~4, 241--254.

\bibitem{Zubkov_1999}A.N. Zubkov, {\it Invariants of an adjoint action of classical groups},
Algebra and Logic {\bf 38} (1999), No.~5, 299--318.

\bibitem{ZubkovII}A.N. Zubkov,  {\it Invariants of mixed representations of quivers II: Defining relations and applications}, J. Algebra Appl. {\bf 4} (2005), No.~3, 287--312.
\end{thebibliography}
\end{document}